\documentclass{amsart}
\usepackage{amsmath,amsthm, amssymb,mathabx, mathrsfs, epsfig, xypic, mathtools}
\usepackage{cite,  mathscinet}
\usepackage[colorlinks]{hyperref}
\usepackage{xr-hyper}
\usepackage{enumitem}
\usepackage{tikz, tikz-cd}
\usetikzlibrary{arrows, calc, patterns, cd, shapes, trees, matrix}

\usepackage{scalerel}[2014/03/10]
\usepackage{stackengine}

\begin{document}

\externaldocument[MRV1-]{"MRVI"}

%macros
\renewcommand{\AA}{\mathbb{A}}

\renewcommand\widetilde[1]{\ThisStyle{%
  \setbox0=\hbox{$\SavedStyle#1$}%
  \stackengine{-.1\LMpt}{$\SavedStyle#1$}{%
    \stretchto{\scaleto{\SavedStyle\mkern.2mu\sim}{.5467\wd0}}{.7\ht0}%
%    .2mu is the kern imbalance when clipping white space
%    .5467++++ is \ht/[kerned \wd] aspect ratio for \sim glyph
  }{O}{c}{F}{T}{S}%
}}

\newcommand{\djunion}{\mathop{\bigsqcup}}
\newcommand{\BB}{\mathbb{B}}
\newcommand{\KK}{\mathbb{K}}
\newcommand{\DD}{\mathbb{D}}
\newcommand{\QQ}{\mathbb{Q}}
\newcommand{\ZZ}{\mathbb{Z}}
\newcommand{\FF}{\mathbb{F}}
\newcommand{\LL}{\mathbb{L}}
\newcommand{\NN}{\mathbb{N}}
\newcommand{\EE}{\mathbb{E}}
\newcommand{\HH}{\mathbb{H}}
\newcommand{\RR}{\mathbb{R}}
\newcommand{\PP}{\mathbb{P}}
\newcommand{\TT}{\mathbb{T}}
\renewcommand{\SS}{\mathbb{S}}
\newcommand{\CC}{\mathbb{C}}
\newcommand{\GG}{\mathbb{G}}
\newcommand{\WW}{\mathbb{W}}

\newcommand{\Oc}{\mathcal{O}}
\newcommand{\Ac}{\mathcal{A}}
\newcommand{\Cc}{\mathcal{C}}
\newcommand{\Ec}{\mathcal{E}}
\newcommand{\Fc}{\mathcal{F}}
\newcommand{\Sc}{\mathcal{S}}
\newcommand{\Gc}{\mathcal{G}}
\newcommand{\Hc}{\mathcal{H}}
\renewcommand{\Rc}{\mathcal{R}}
\newcommand{\Pc}{\mathcal{P}}
\renewcommand{\Lc}{\mathcal{L}}
\renewcommand{\Uc}{\mathcal{U}}
\renewcommand{\Yc}{\mathcal{Y}}

\renewcommand{\Mc}{\mathcal{M}}
\newcommand{\Vc}{\mathcal{V}}

\newcommand{\coker}{\mathrm{coker}}
\newcommand{\gf}{\mathfrak{g}}
\newcommand{\pf}{\mathfrak{p}}
\newcommand{\mf}{\mathfrak{m}}
\newcommand{\qf}{\mathfrak{q}}
\newcommand{\kf}{\mathfrak{k}}
\newcommand{\hf}{\mathfrak{h}}
\newcommand{\Xf}{\mathfrak{X}}

\newcommand{\tensor}{\otimes}
\newcommand{\extp}{\mathchoice{{\textstyle\bigwedge}}%
    {{\bigwedge}}%
    {{\textstyle\wedge}}%
    {{\scriptstyle\wedge}}}
\newcommand{\comptensor}{\widehat{\otimes}}

\newcommand{\cont}{\mathrm{cont}}
\newcommand{\Newt}{\mathrm{Newt}}

\newcommand{\actsonr}{\mathrel{\reflectbox{$\righttoleftarrow$}}}
\newcommand{\actsonl}{\mathrel{\reflectbox{$\lefttorightarrow$}}}

\newcommand{\isoeq}{\cong}
\newcommand{\cech}{\vee}
\newcommand{\dsum}{\mathop{\oplus}}

\newcommand{\Tr}{\mathrm{Tr}}
\newcommand{\Gal}{\mathrm{Gal}}

\newcommand{\Mod}{\mathrm{Mod}}
\newcommand{\Rep}{\mathrm{Rep}}
\newcommand{\Hecke}{\mathrm{Hecke}}
\newcommand{\Fil}{\mathrm{Fil}}
\newcommand{\Qlog}{\mathrm{Qlog}}
\newcommand{\Image}{\mathrm{Im}}
\newcommand{\Ker}{\mathrm{ker}}
\newcommand{\Id}{\mathrm{Id}}
\newcommand{\std}{\mathrm{std}}
\newcommand{\dR}{\mathrm{dR}}
\newcommand{\Gr}{\mathrm{Gr}}
\newcommand{\Aut}{\mathrm{Aut}}
\newcommand{\Pic}{\mathrm{Pic}}
\newcommand{\Div}{\mathrm{Div}}
\newcommand{\Frob}{\mathrm{Frob}}
\newcommand{\Proj}{\mathrm{Proj}}
\newcommand{\Spa}{\mathrm{Spa}}
\newcommand{\Spf}{\mathrm{Spf}}
\newcommand{\Spec}{\mathrm{Spec}}
\newcommand{\Hom}{\mathrm{Hom}}
\newcommand{\Isom}{\mathrm{Isom}}
\newcommand{\proet}{\mathrm{pro\acute{e}t}}
\newcommand{\Lie}{\mathrm{Lie}}
\newcommand{\perf}{\mathrm{perf}}
\renewcommand{\Hecke}{\mathrm{Hecke}}

\newcommand{\IC}{\mathrm{IC}}

\newcommand{\Map}{\mathrm{Map}}
\newcommand{\Ext}{\mathrm{Ext}}

\newcommand{\Homc}{\mathrm{Hom}_\mathrm{cont}}

\newcommand{\height}{\mathrm{ht}}
\newcommand{\un}{\mathrm{un}}
\newcommand{\et}{\mathrm{\acute{e}t}}
\newcommand{\profet}{\mathrm{prof\acute{e}t}}
\newcommand{\an}{\mathrm{an}}
\newcommand{\cris}{\mathrm{cris}}
\newcommand{\cyc}{\mathrm{cyc}}
\newcommand{\ad}{\mathrm{ad}}
\newcommand{\cl}{\mathrm{cl}}
\newcommand{\coh}{\mathrm{coh}}
\newcommand{\GL}{\mathrm{GL}}
\newcommand{\union}{\bigcup}

\newcommand{\rank}{\mathrm{rank}}
\newcommand{\ord}{\mathrm{ord}}
\newcommand{\Cont}{\mathrm{Cont}}
\newcommand{\Bun}{\mathrm{Bun}}
\newcommand{\piF}{\varpi_F}
\newcommand{\Irred}{\mathrm{Irred}}

\newcommand{\Ad}{\mathrm{Ad}}

\newcommand{\tr}{\mathrm{tr}}

\newcommand{\bbL}{\mathbb{L}}

\newcommand{\calQ}{\mathcal{Q}}
\newcommand{\bbN}{\mathbb{N}}
\newcommand{\bbA}{\mathbb{A}}
\newcommand{\bbZ}{\mathbb{Z}}

\newcommand{\Conf}{\mathrm{Conf}}
\newcommand{\Sym}{\mathrm{Sym}}
\newcommand{\End}{\mathrm{End}}

\DeclarePairedDelimiter\floor{\lfloor}{\rfloor}

%theorem environments
\theoremstyle{plain}

%the following are labelled without section numbers
\newtheorem{maintheorem}{Theorem} 
\renewcommand{\themaintheorem}{\Alph{maintheorem}} 
\newtheorem{maincorollary}[maintheorem]{Corollary}
\newtheorem{mainconjecture}[maintheorem]{Conjecture}

%unlabelled 
\newtheorem*{theorem*}{Theorem}

%labeled with subsection numbers
\newtheorem{theorem}{Theorem}[subsection]
\newtheorem{corollary}[theorem]{Corollary}
\newtheorem{conjecture}[theorem]{Conjecture}
\newtheorem{proposition}[theorem]{Proposition}
\newtheorem{lemma}[theorem]{Lemma}

\numberwithin{equation}{theorem}

\theoremstyle{definition}

%labeled with subsection numbers
\newtheorem{example}[theorem]{Example}
\newtheorem{definition}[theorem]{Definition}
\newtheorem{remark}[theorem]{Remark}
\newtheorem{question}[theorem]{Question}

%unlabeled
\newtheorem{principle*}{Principle}

\title[M.R.V. and rep. stability II: Hypersurface sections]{Motivic random variables and representation stability II: Hypersurface sections}
\author{Sean Howe}
\address{Department of Mathematics, University of Utah, Salt Lake City, UT 84112}
\email{sean.howe@utah.edu}
\thanks{This material is based upon work supported by the National Science Foundation under Award
No. DMS-1704005.}

\maketitle

\tableofcontents

\newcommand{\motcomp}{\widehat{\Mc_\LL}} 
\newcommand{\mot}{\Mc}
\newcommand{\Var}{\mathrm{Var}}
\newcommand{\van}{\mathrm{van}}
\newcommand{\Loc}{\mathrm{Loc}}
\newcommand{\VHS}{\mathrm{VHS}}
\newcommand{\HS}{\mathrm{HS}}
\newcommand{\MHM}{\mathrm{MHM}}

\newcommand{\VanQ}{\Vc_{\van,\QQ}}
\newcommand{\VanQl}{\Vc_{\van,\QQ_l}}

\newcommand{\old}{\mathrm{old}}
\newcommand{\Pow}{\mathrm{Pow}}

\begin{abstract}
We prove geometric and cohomological stabilization results for the universal smooth degree $d$ hypersurface section of a fixed smooth projective variety as $d$ goes to infinity. We show that relative configuration spaces of the universal smooth hypersurface section stabilize in the completed Grothendieck ring of varieties, and deduce from this the stabilization of the Hodge Euler characteristic of natural families of local systems constructed from the vanishing cohomology. We prove  explicit formulas for the stable values using a probabilistic interpretation, along with the natural analogs in point counting over finite fields. We explain how these results provide new geometric examples of a weak version of representation stability for symmetric, symplectic, and orthogonal groups. This interpretation of representation stability was studied in the prequel \cite{Howe-MRV1} for configuration spaces.
\end{abstract}

\section{Introduction} 

In this work, we prove geometric and cohomological stabilization results for the universal smooth degree $d$ hypersurface section of a fixed smooth projective variety as $d$ goes to infinity, bringing together three important ideas in the field of cohomological stability: motivic stabilization as introduced by Vakil-Wood \cite{VakilWood-Discriminants}, representation stability as introduced by Church \cite{Church-HomStab} and Church-Farb \cite{ChurchFarb-RepStab}, and connections with arithmetic statistics as discussed, e.g., in \cite{CEF-RepStabFinField}. Indeed, our stabilization results in the prequel \cite{Howe-MRV1} can be thought of as a motivic version of representation stability for configuration spaces, with explicit stable values provided by a probabilistic interpretation motivated by arithmetic statistics over finite fields. Our results in this work go further, giving a motivic version of representation stability for symplectic and orthogonal groups (as well as new instances for symmetric groups) in a setting where there are no other known representation stability results. A key aspect of our theory is the introduction of motivic random variables, which provide a useful way of organizing and understanding the stable values by using probability theory and transporting intuition from characteristic $p$, where point counting methods apply, to characteristic zero. 

\subsection{Cohomological stabilization}\label{subsec:IntroCohomStab}
Let $Y$ be a polarized smooth projective variety -- i.e. a smooth projective variety equipped with a very ample line bundle $\Lc$ giving a closed embedding $Y \rightarrow \PP(\Gamma(Y, \Lc)^*)$.  Let $U_d$ be the space of smooth hypersurface sections of degree $d$ of $Y$ (with respect to this embedding). The non-constant part of the cohomology of the universal smooth hypersurface section $Z_d / U_d$ gives rise to a local system, the \emph{vanishing cohomology} $\VanQ$ on $U_d(\CC)$. The corresponding monodromy representation is to a symmetric, orthogonal, or symplectic group, depending on whether $\dim Y - 1$ is zero, even and positive, or odd -- we will refer to this group as the \emph{algebraic monodromy group}.

If $\pi$ is a representation of the algebraic monodromy group, then by composing the monodromy representation with $\pi$ we obtain a new local system $\VanQ^\pi$ on $U_d(\CC)$. In any of these situations, a partition $\sigma$ gives rise to a family of irreducible representations $\pi_{\sigma,d}$ of the algebraic monodromy groups for $U_d$ (by the theory of Young tableaux for symmetric groups, as in the theory of representation stability \cite{ChurchFarb-RepStab}, and by highest weight theory for symplectic and orthogonal groups -- cf. Section \ref{sec:RepTheory}). 

Each of the local systems 
\[ \VanQ^{\pi_{\sigma,d}} \]
is equipped with a natural variation of Hodge structure, and thus its cohomology is equipped with a mixed Hodge structure\footnote{To simplify some arguments we use the theory of Arapura \cite{Arapura-LerayMotivic} rather than that of Saito \cite{Saito-MHM1, Saito-MHM2} to produce this mixed Hodge structure -- cf. Subsection \ref{subsec:ProbVHS}.}. In particular, denoting the weight filtration by $W$,
\[ \Gr_W H^i_c(U_d(\CC), \VanQ) \]
is a direct sum of polarizable Hodge structures. 

We define $K_0(\HS)$ to be the Grothendieck ring of polarizable Hodge structures, which is the quotient of the free $\ZZ-$module with basis given by isomorphism classes $[V]$ of polarizable $\QQ-$Hodge structures $V$ by the relations $[V_1 \dsum V_2]=[V_1]+[V_2]$. It is a ring with $[V_1]\cdot[V_2]=[V_1 \otimes V_2]$.

We denote by
\[ \QQ(-1) = H^2_c(\AA^1), \]
the Tate Hodge structure of weight 2, and $\QQ(n)=\QQ(-1)^{\otimes -n}$. 

By the above considerations, we obtain for each $d$ a compactly supported Euler characteristic 
\[ \chi_\HS (H_c^\bullet(U_d(\CC), \VanQ^{\pi_{\sigma,d}})) := \sum_i (-1)^i [\Gr_W H^i_c(U_d(\CC), \VanQ^{\pi_{\sigma,d}})] \in K_0(\HS) \]
Our first main theorem states that this class stabilizes as $d\rightarrow \infty$ in the completion $\widehat{K_0(\HS)}$ of $K_0(\HS)$ for the weight filtration.

\begin{maintheorem}\label{thm:StabCohomHodge}
If $Y/\CC$ is a polarized smooth projective variety of dimension $n\geq 1$, then for any partition $\sigma$, 
\[ \lim_{d\rightarrow \infty} \frac{\chi_\HS (H_c^\bullet(U_d(\CC), \VanQ^{\pi_{\sigma,d}}))}{[\QQ(-\dim U_d)]} \]
exists in $\widehat{K_0(\HS)}.$ Moreover, it can be expressed by an explicit universal formula as a limit of elements in the subring of $K_0(\HS)$ generated by symmetric powers of the cohomology groups of $Y$ and $\QQ(1)$.  
\end{maintheorem}

The formula is universal in the sense that for each $n$ and $\sigma$ there is a power series in a single variable $t$ with coefficients given by polynomials in the symbols $[\Sym^i H^j(\bullet, \QQ)]$, $i\geq 0,\; 0 \leq j\leq 2n$, such that the limit is obtained for any $Y$ of dimension $n$ by substituting $\QQ(1)$ for $t$ and $[\Sym^i H^j(Y, \QQ)]$ for $[\Sym^i H^j(\bullet, \QQ)]$. This power series is even a rational function, and can be made explicit: Both Theorem~\ref{thm:StabCohomHodge} and its point-counting analog, Theorem~\ref{thm:StabCohomPointCount} below, are deduced from corresponding geometric stabilization results (Theorems \ref{thm:StabGeomMot} and \ref{thm:StableGeomPointCount} below), and the universal formulas are more natural in the geometric setting where they have a probabilistic interpretation. We give the geometric universal formulas below in Theorems \ref{thm:StabGeomMot} and \ref{thm:StableGeomPointCount}, and in Appendix \ref{appendix:AlgAndComp} we extract from the proofs an algorithm to obtain the cohomological universal formulas of Theorems \ref{thm:StabCohomHodge} and \ref{thm:StabCohomPointCount} for a given $\sigma$ and $n$.

\begin{example}\label{example:IntroPrimCohomVanish}
If $Y=\PP^n$, then $\VanQ$ is the local system of primitive cohomology coming from the universal smooth hypersurface. In this case, we deduce (cf. Example \ref{example:AppdxPrimCohomVanish} in Appendix \ref{appendix:AlgAndComp}) that
\[ \lim_{d\rightarrow \infty} \frac{\chi_\HS (H_c^\bullet(U_d(\CC), \VanQ))}{[\QQ(-\dim U_d)]} = 0. \]
We also obtain the point counting analog.  One could view this as evidence that the individual cohomology groups $H^i(U_d(\CC), \VanQ)$ are themselves stably trivial, though it is also possible, e.g., that the cohomology groups stabilize with cancellation between weights in different degrees (in upcoming work \cite{Howe-MRVRM} we make a general cohomological stabilization conjecture compatible with Theorems \ref{thm:StabCohomHodge}, \ref{thm:StabCohomPointCount}, \ref{thm:StabGeomMot} and \ref{thm:StableGeomPointCount}). We note that in this case $H^1$ is known to be stably trivial by Nori's connectivity theorem \cite[Corollary 4.4]{Nori-Connectivity}\footnote{We thank Bhargav Bhatt for pointing out the connection between our results and Nori's connectivity theorem.}. 
\end{example}

\begin{remark} \label{rmk:StabCohomHodgeInitial} \hfill\\
\textbf{1.} Theorem \ref{thm:StabCohomHodge} is new except for the trivial local system (i.e. the cohomology of $U_d(\CC)$ itself), corresponding to $\sigma=\emptyset$, where it is due to Vakil and Wood \cite{VakilWood-Discriminants}. The methods of \cite{VakilWood-Discriminants} will play an important role in our proof. 

\noindent\textbf{2.} Because the $U_d$ are smooth but not proper, the Hodge Euler characteristic can potentially contain less information than the cohomology groups (with mixed Hodge structures) themselves. Nevertheless, in practice it seems a large amount of information is retained: for example, for $Y=\PP^n$ and the trivial local system, Tommasi \cite{Tommasi-Stable} has shown the individual cohomology groups also stabilize and from her computation one finds that there is no cancellation between degrees. We refer the reader to \cite[1.2]{VakilWood-Discriminants} for more on this Occam's razor principle for Hodge structures. 
\end{remark}

The local systems $\VanQ^{\pi_{\sigma,d}}$ also have $l$-adic incarnations $\VanQl^{\pi_{\sigma,d}}$ in \'{e}tale cohomology, and we study these over finite fields. We obtain the point-counting analog

\begin{maintheorem}\label{thm:StabCohomPointCount}
If $Y /\FF_q$ is a polarized smooth projective variety of dimension $n\geq1$ and $\sigma$ is a partition, then
\[ \lim_{d \rightarrow \infty} q^{-\dim U_d}\sum_i (-1)^i \Tr \Frob_q \actsonr H_c^\bullet (U_{d,\overline{\FF_q}}, \VanQl^{\pi_{\sigma, d}}) \]
exists in $\QQ$ (here the limit is of elements of $\QQ$ in the archimedean topology). Furthermore, for a fixed $\sigma$ and dimension $n$, the limit is given by an explicit, computable universal formula. The universal formula is a rational function of $q$ and symmetric functions of the eigenvalues of Frobenius acting on the cohomology of $Y$. 
\end{maintheorem}

\begin{remark} For the trivial local system, Theorem \ref{thm:StabCohomPointCount} is due to Poonen \cite{Poonen-Bertini}. Furthermore, one of the main technical inputs in our proof of Theorem \ref{thm:StabCohomPointCount} is Poonen's sieving Bertini theorem with Taylor conditions \cite[Theorem 1.2]{Poonen-Bertini} (cf. Remark \ref{rmk:PoonenAsympInd} below). 
\end{remark}

\subsection{Motivic stabilization}\label{subsec:IntroMotStab}

Returning to the case of $Y/\CC$, our goal now is to systematically understand the stabilization of natural varieties over $U_d$ produced from the universal family $Z_d$ (e.g., relative configuration spaces). 

To study this, we will use the notion of a relative Grothendieck ring of varieties. For $S/\CC$ a variety, we denote by $K_0(\Var/S)$ the ring spanned by isomorphism classes $[X/S]$ of varieties $X/S$ (i.e. maps of varieties $f: X \rightarrow S$) modulo the relations 
\[ [X/S] = [X\backslash Z / S] + [Z/S] \]
for any closed subvariety $Z \subset X$. It is a ring with 
\[ [X_1/S]\cdot[X_2/S]=[X_1 \times_S X_2 / S].\]
We will write $K_0(\Var)$ for $K_0(\Var/\CC)$. 

For any $S/\CC$ there is a natural \emph{pre-$\lambda$} structure on the Grothendieck ring $K_0(\Var / S)$ which gives an efficient way to work with constructions such as relative symmetric powers and relative configuration spaces (cf. \cite[Subsection  \ref{MRV1-subsec:GrothPow}]{Howe-MRV1}). This pre-$\lambda$ structure can be interpreted as a set-theoretic pairing
\[ (\; ,\;) : \Lambda \times K_0(\Var / S) \rightarrow K_0(\Var / S) \]
where 
\[ \Lambda = \ZZ[h_1, h_2, h_3,...] \]
is the ring of symmetric functions \cite{Macdonald-SymmetricFunctions} (the $h_i$ are the complete symmetric functions). If we fix the second variable, the pairing gives a ring homomorphism 
\[ \Lambda \rightarrow K_0(\Var / S) \]
characterized by
\[ (h_k, [T/S] ) = [\Sym^k_{S} T / S] \]
where the subscript on the symmetric power denotes that it is taken relative to $S$. 

Our geometric stabilization will take place in 
\[ \motcomp := \widehat{K_0(\Var)[\LL^{-1}]} \]
where $\LL=[\AA^1]$ and the completion is with respect to the dimension filtration (cf. \cite{VakilWood-Discriminants}). For $x \in K_0(\Var/S)$, we will denote by $x_{\motcomp}$ the element of $\motcomp$ obtained by forgetting the structure morphism to $S$ (i.e. by applying the map $[T/S] \mapsto [T]$). 

As a first approximation, our motivic stability result says

\begin{theorem*}[first version of Theorem \ref{thm:StabGeomMot}]
If $Y/\CC$ is a polarized smooth projective variety of dimension $n\geq 1$, then, for any symmetric function $f \in \Lambda$, 
\[ \lim_{d \rightarrow \infty} \frac{(f, [Z_d/U_d])_{\motcomp}}{\LL^{\dim U_d}} \]
exists in $\motcomp$. 
\end{theorem*}

We will refine this statement with explicit formulas in Theorem \ref{thm:StabGeomMot} below. To motivate these, it is helpful to first consider the point-counting analog, which we do below. First, however, we give some examples and remarks:

\begin{remark}\label{rmk:VWMotZeta}
For $f=1$, this result is due to Vakil-Wood \cite{VakilWood-Discriminants}, who compute
\[ \lim_{d\rightarrow \infty}\frac{[U_d]}{\LL^{\dim U_d}}=\zeta_Y(n+1)^{-1}. \]
Here $\zeta_Y(n+1)$ is obtained by substituting $t=\LL^{-(n+1)}$ in the Kapranov zeta function 
\[ Z_{Y}(t)= 1 + [Y]\cdot t + [\Sym^2 Y] \cdot t^2 + ... \in 1 + t K_0(\Var)[[t]]. \]
\end{remark}

\begin{example}[\emph{Relative generalized configuration spaces stabilize}]\label{example:RelConfStab}
For any generalized partition $\tau=a_1^{l_1}\cdot...\cdot a_m^{l_m}$ we denote by $\Conf^\tau_{S} T/S$ the relative generalized configuration space which parameterizes collections of $\sum l_i$ distinct points lying in a fiber of $T \rightarrow S$, of which $l_i$ are labeled by $a_i$ for each $1\leq i\leq m$. There is a unique $c_\tau \in \Lambda$ such that for all $S$ and all $T/S$,
\[ (c_\tau, [T/S] ) = [\Conf^\tau_{S} T / S]. \]
Indeed, \cite[3.19]{VakilWood-Discriminants} gives an explicit formula for the class of a generalized configuration space in terms of the classes of symmetric powers and shows that this formula is unique; this translates to a formula for $c_\tau$ in terms of complete symmetric functions. 
 
In particular, applied to $c_\tau$ our theorem shows
\[ \lim_{d \rightarrow \infty} \frac{[\Conf^\tau_{U_d} Z_d]}{\LL^{\dim U_d}} \]
exists in $\motcomp$. (In fact, we will \emph{prove} Theorem \ref{thm:StabGeomMot} below by showing this statement and using that the $c_\tau$ form a basis for $\Lambda$). 
\end{example}

\begin{example}
For $Y=\PP^n$ and $f=h_1$, so that we are considering $Z_d$ the universal family of smooth hypersurfaces in $\PP^n$, we obtain (using Remark \ref{rmk:VWMotZeta} and Example \ref{example:AvgHSMot} below)
\[ \lim_{d \rightarrow \infty} \frac{[Z_d]}{\LL^{\dim U_d}} = [\PP^{n-1}]\cdot \zeta_{\PP^n}(n+1)^{-1} \]
where $\zeta_{\PP^{n}}(n+1)$ is the Kapranov zeta function of $\PP^n$ evaluated at $\LL^{-(n+1)}$. 
\end{example}

\begin{remark}\label{rem:DedCohomGeom}
The cohomological Theorem~\ref{thm:StabCohomHodge} is deduced from the geometric Theorem~\ref{thm:StabGeomMot} below by taking $f$ to be a (symmetric, symplectic, or orthogonal) Schur polynomial $s_\sigma$ then carefully modifying the result to remove the contribution of the constant part of the cohomology of $Z_d$ (which can be described explicitly in terms of the cohomology of $Y$). The key tool that makes this extraction possible is the fact that the Adams operations on a pre-$\lambda$ ring, given by pairing with power sum symmetric functions, are additive. 
\end{remark} 

\subsection{Probabilistic interpretation}\label{subsec:ProbInterpretation}

In Theorem \ref{thm:StableGeomPointCount} below we give a point-counting geometric stabilization theorem with explicit formulas for the limit stated using probabilistic language. Motivated by the explicit formulas in the point-counting setting, we then state our full motivic stabilization result as Theorem \ref{thm:StabGeomMot} below. 

\newcommand{\PConf}{\mathrm{PConf}}

For $Y/ \FF_q$ a polarized smooth projective variety, we consider the subring $K_0(\Loc_{\QQ_l} U_d)'$ of the Grothendieck ring of lisse $l$-adic sheaves consisting of virtual sheaves with integral characteristic series of Frobenius at every closed point. It is a pre-$\lambda$ ring, and it admits an algebraic probability measure $\mu_d$ (in the sense of \cite[Section \ref{MRV1-sec:Probability}]{Howe-MRV1}) with values in $\QQ$ where the expectation $\EE_{\mu_d}$ is given by averaging the traces of Frobenius over $\FF_q$-rational points (or, equivalently by the Grothendieck-Lefschetz formula, by computing the alternating sum of traces on the compactly supported cohomology then dividing by $\# U_d(\FF_q)$). 
	
	In $K_0(\Loc_{\QQ_l} U_d)'$, there is a class $[Z_d / U_d]$ given by the alternating sum of the cohomology local systems of $Z_d$. We can view $K_0(\Loc_{\QQ_l} U_d)'$ as a ring of \emph{motivic random variables} lifting the ring of classical random variables on the discrete probability space $U_d(\FF_q)$ with uniform distribution, and from this perspective the random variable $[Z_d/U_d]$ lifts the classical random variable assigning to a smooth hypersurface section $u \in U_d(\FF_q)$ the number of $\FF_q$ points on $Z_{d,u}$. 

Let 
\[ p_k' := \frac{1}{k} \sum_{d|k} \mu(d/k) p_k \in \Lambda_\QQ \]  
be the Mobius inverted power sum polynomials. We can view $(p_k', [Z_d/U_d])$ as a motivic lift of the classical random variable assigning to a smooth hypersurface section $u \in U_d(\FF_q)$ the number of degree $k$ closed points on $Z_{d,u}$. 

\begin{maintheorem}\label{thm:StableGeomPointCount}
Let $Y /\FF_q$ be a polarized smooth projective variety of dimension $n\geq 1$.  Then, for a formal variable $t$, 
\[ \lim_{d \rightarrow \infty} \EE_{\mu_d}\left[ (1+t)^{(p_k', [Z_d/U_d])} \right] = \left(1+ \frac{q^{nk} - 1}{q^{(n+1)k}-1} t \right)^{\# \textrm{ closed points of degree $k$ on } Y} \]
and the $(p_k', [Z_d/U_d])$ are asymptotically independent for distinct $k$, i.e. for formal variables $t_1, t_2,...,t_m$,
\[ \lim_{d\rightarrow \infty} \EE_{\mu_d}\left [ \prod_k (1+t_k)^{(p_k', [Z_d/U_d]} \right] = \prod_k \lim_{d\rightarrow \infty} \EE_{\mu_d}\left [ (1+t_k)^{(p_k', [Z_d/U_d])} \right]. \]
Here exponentiation is interpreted in terms of the standard power series
\[ (1+t)^a =\exp(\log(1+t)\cdot a) = \sum_i \binom{a}{i} t^i, \]
and $\EE_{\mu_d}$ and limits are applied individually to each coefficient of a power series.
  
In particular, for any symmetric function $f \in \Lambda$,
\[ \lim_{d\rightarrow \infty} \EE_{\mu_d}[(f,[Z_d/U_d])] \]
exists in $\QQ$, and can be computed explicitly by expressing $f$ as a polynomial in the $p_k'$ and applying asymptotic independence plus the explicit distributions above. 
\end{maintheorem}

\begin{remark}\label{rmk:PoonenAsympInd}
The probabilities in Theorem \ref{thm:StableGeomPointCount} have a simple geometric origin: Each closed point $y$ of degree $k$ in  $Y$ defines an indicator Bernoulli random variable on $U_d(\FF_q)$ that is $1$ at $u \in U_d(\FF_q)$ if $y$ is contained in the fiber $Z_{d,u}$ and 0 otherwise. These random variables are asymptotically independent each with asymptotic probability of being $1$ determined by the proportion of the number of smooth first order Taylor expansions vanishing at~$y$, ${q^{nk}-1}$, to the total number of smooth first order Taylor expansions at $y$, ${q^{(n+1)k}-1}$. The asymptotic independence of these Bernoulli random variables is a simple consequence of a result of Poonen \cite[Theorem 1.2]{Poonen-Bertini}, and using this it is straightforward to prove Theorem~\ref{thm:StableGeomPointCount}. 
\end{remark}

\begin{example}[\emph{The average smooth hypersurface in $\PP^n$, point counting version}]\label{example:AvgHSPC}
Taking $Y=\PP^n$ and $f=p_1$ in Theorem \ref{thm:StableGeomPointCount}, we find 
\begin{eqnarray*}
\lim_{d \rightarrow \infty} \frac{\sum_{u \in U_d(\FF_q)} \# Z_{d,u}(\FF_q)}{\# U_d(\FF_q)} =\lim_{d\rightarrow \infty} \EE_{\mu_d}\left[ [Z_d/U_d] \right ] & = & \# \PP^{n}(\FF_q) \cdot \frac{q^n -1}{q^{n+1}-1}\\ 
& = & \# \PP^{n-1}(\FF_q)
\end{eqnarray*}
Thus, in this sense the average smooth hypersurface in $\PP^n$ is $\PP^{n-1}$. 
\end{example}

\begin{remark}
Explicit formulas in Theorem \ref{thm:StabCohomPointCount} can be deduced from Theorem \ref{thm:StableGeomPointCount} as in Remark \ref{rem:DedCohomGeom}. We carefully describe an algorithm for this in Appendix \ref{appendix:AlgAndComp}. 
\end{remark}

We now return to $Y / \CC$. Because $U_d$ is an open in affine space, $[U_d]$ is invertible in $\motcomp$\footnote{We thank Jesse Wolfson for pointing out this fact to us.}. Thus we obtain an algebraic probability measure on $K_0(\Var / U_d)$ with values in $\motcomp$, characterized by the expectation function 
\[ \EE_{\mu_d} [ x ] = \frac{x_{\motcomp}}{[U_d]}. \]
The motivic stabilization stated in subsection \ref{subsec:IntroMotStab} can then be refined to give explicit formulas, analogous to Theorem \ref{thm:StableGeomPointCount}: 

\begin{maintheorem}\label{thm:StabGeomMot}
Let $Y/\CC$ be a polarized smooth projective variety of dimension $n\geq 1$. For $t$ a formal variable,
\[ \lim_{d \rightarrow \infty} \EE_{\mu_d}\left[ (1+t)^{(p_k', [Z_d/U_d])} \right] = \left(1+ \frac{\LL^{nk} - 1}{\LL^{(n+1)k}-1} t \right)^{(p_k', [Y])} \]
in $\motcomp \otimes \QQ$ and the $(p_k', [Z_d/U_d])$ are asymptotically independent for distinct $k$,  i.e. for formal variables $t_1, t_2,...,t_m$,
\[ \lim_{d\rightarrow \infty} \EE_{\mu_d}\left [ \prod_k (1+t_k)^{(p_k', [Z_d/U_d])} \right] = \prod_k \lim_{d\rightarrow \infty} \EE_{\mu_d}\left [ (1+t_k)^{(p_k', [Z_d/U_d])} \right]. \]
Here exponentiation is interpreted in terms of the standard power series
\[ (1+t)^a =\exp(\log(1+t)\cdot a) = \sum_i \binom{a}{i} t^i, \]
and $\EE_{\mu_d}$ and limits are applied individually to each coefficient of a power series.

Furthermore, for any symmetric function $f \in \Lambda$,
\[ \lim_{d\rightarrow \infty} \EE_{\mu_d}[(f,[Z_d/U_d])] \]
exists in $\motcomp$. It can be computed explicitly in $\motcomp \otimes \QQ$ by expressing $f$ as a polynomial in the $p_k'$ and applying asymptotic independence plus the explicit distributions above. 
\end{maintheorem}

\begin{remark} As in \cite{Howe-MRV1}, the significance of the terms $(p_k', \bullet)$ appearing in Theorem \ref{thm:StabGeomMot} are that they give the exponents in the naive Euler product for the Kapranov zeta function. It is in this sense that they are natural motivic avatars for the number of closed points of a fixed degree on a variety over a finite field. We note that the $p_k'$ are in $\Lambda_\QQ$, and typically not in $\Lambda$, which explains the need to tensor with $\QQ$.
\end{remark}

\begin{example}\label{example:ExplicitFormula-p1} 
We denote by $\PConf^k_{U_d} Z_d / U_d$ the relative configuration space of $k$ distinct ordered points in $Z_d$. In the notation of Example \ref{example:RelConfStab}, 
\[ \PConf^k_{U_d} Z_d= \Conf^{a_1 \cdot ... \cdot a_k}_{U_d} Z_d. \]
We have the identity
\begin{eqnarray*} [\PConf^k_{U_d} Z_d / U_d ] & = & ([Z_d/U_d])\cdot([Z_d/U_d]-1)\cdot...([Z_d/U_d]-k+1) \\
&=& \left( (p_1)(p_1 - 1)...(p_1 - k +1), [Z_d / U_d] \right ). 
\end{eqnarray*}
Thus, from Theorem \ref{thm:StabGeomMot}, we obtain
\[ \lim_{d \rightarrow \infty} \frac{[\PConf^k_{U_d} Z_d]}{[U_d]} =  \PConf^k Y \cdot \left(\frac{\LL^{n}-1}{\LL^{n+1}-1}\right)^k \]
in $\motcomp\otimes\QQ$ (in fact, adapting the proof one finds this holds already in $\motcomp$). Using Theorem \ref{thm:StableGeomPointCount}, we also obtain the corresponding point-counting result. 
\end{example}

\begin{example}[\emph{The average smooth hypersurface in $\PP^n$, motivic version.}]\label{example:AvgHSMot}
Taking $Y=\PP^n$ and $k=1$ in Example \ref{example:ExplicitFormula-p1}, we obtain
\begin{eqnarray*} \lim_{d\rightarrow \infty} \frac{[Z_d]}{[U_d]} & = &[\PP^{n}] \cdot \left(\frac{\LL^{n}-1}{\LL^{n+1}-1}\right) \\
& = & [\PP^{n-1}]
\end{eqnarray*}
Thus, in this sense the average smooth hypersurface in $\PP^n$ is $\PP^{n-1}$ (cf. Example \ref{example:AvgHSPC} for the point counting version). 
\end{example}

We prove Theorem \ref{thm:StabGeomMot} by extending the methods of Vakil-Wood \cite{VakilWood-Discriminants} to show that the limit exists and is equal to a more complicated explicit formula in terms of configuration spaces of $Y$, then comparing with a generating function constructed via the power structure on the Grothendieck ring of varieties. In our study of configuration spaces \cite{Howe-MRV1}, we also used a generating function constructed via the power structure, however, the present case is more complicated -- the form of the probabilities which appear means that the series which we exponentiate cannot have effective coefficients. In order to compute with such a series, we use the motivic Euler products of Bilu \cite{bilu:thesis, bilu:thesis-article}.

\begin{remark} 
We note that in an earlier version of this article which was circulated and appeared on arXiv, the explicit stable values of Theorem \ref{thm:StabGeomMot} were stated as a conjecture and proven only in the special case of Example \ref{example:ExplicitFormula-p1}; while the generating function strategy was outlined, we could not prove the necessary formula for a power of a non-effective series. We arrived at a full proof only after discussions of \cite{bilu:thesis} with Bilu. In joint work in progress \cite{howe-bilu:euler-and-stablization}, we use motivic Euler products to give a simultaneous generalization of Theorem \ref{thm:StabGeomMot} and the results on hypersurface sections of \cite{VakilWood-Discriminants} that is closer in spirit to the proofs of the point-counting analogs. 
\end{remark}

\subsection{Connections with representation stability and the prequel}\label{subsec:ConnRepStab}
As noted in the introduction to the prequel \cite{Howe-MRV1}, one consequence of representation stability for configuration spaces (as studied in \cite{Church-HomStab, ChurchFarb-RepStab}) is cohomological stabilization for certain families of local systems factoring through the monodromy representation. Our results give an analog of this for the spaces $U_d$, at least at the level of the Hodge Euler characteristic: Generically the image of the monodromy representation is a lattice in the algebraic monodromy group, and thus, by superrigidity, when we consider algebraic representations we are really considering all natural local systems which factor through the monodromy representation. We view Theorems \ref{thm:StabCohomHodge} and \ref{thm:StabCohomPointCount} together as strong evidence that the individual cohomology groups of the local systems $\VanQ^{\pi_{\sigma,d}}$ themselves stabilize, and that the unstable cohomology has sub-exponential growth (this kind of relation between point-counting and cohomological stability is studied in~\cite{FarbWolfson-EtaleStability}).

As with the analogous results of the prequel \cite{Howe-MRV1} for configurations spaces, the form of Theorems \ref{thm:StableGeomPointCount} and \ref{thm:StabGeomMot} suggests that the stabilization is of modules over the (conjectural) stable cohomology of $U_d$. In \cite{Howe-MRVRM} we make an explicit conjecture to this effect motivated by these results and an analogy with moduli of curves. Note that hypersurface sections are more difficult to analyze using ideas from representation stability than configuration spaces because there are no obvious stabilizing maps, however, it is natural to hope that with a more sophisticated approach one might find a richer structure explaining and strengthening our results.

\subsection{Related work}

Besides the works of Poonen \cite{Poonen-Bertini} and Vakil and Wood \cite{VakilWood-Discriminants} discussed above, we highlight some other related threads in the literature:

\subsubsection{Plane curves}
Bucur, David, Feigon, and Lalin \cite{BDFL-FluctuationsPlaneCurves} use Poonen's sieve with refined control over the error term to show that the distribution of the number of degree 1 points on a random smooth plane curve over $\FF_q$ is asymptotically Gaussian (letting both $q$ and $d$ go to infinity in a controlled way). They use the same probabilistic interpretation of Poonen's results that we adopt here, and in particular Theorem 1.1 of \cite{BDFL-FluctuationsPlaneCurves} is a refined version of the statement that the number of degree 1 points on a smooth curve of degree $d$ over $\FF_q$ is asymptotically distributed as a sum of $\# \PP^2(\FF_q)$ independent Bernoulli random variables as $d \rightarrow \infty$. 

\subsubsection{Moduli of curves}
Achter, Erman, Kedlaya, Wood, and Zurieck-Brown \cite{AEKWZB-HeursticRandomCurves} use results on the stable cohomology of moduli of genus $g$ curves with $n$ marked points $\Mc_{g,n}$ to make conjectures about the asymptotic distribution of degree 1 points on curves over $\FF_q$ as the genus $g \rightarrow \infty$. They give a heuristic explanation of why the unstable cohomology should not contribute and deduce an asymptotic distribution from the stable cohomology under this assumption. Their main observation about the stable cohomology can be rephrased by saying that for $C_g / M_g$ the universal genus $g$ curve, $[C_g/M_g]$ is asymptotically a Poisson random variable with respect to the Hodge measure. In \cite{Howe-MRVRM}, we generalize this by showing that the natural motivic random variables $(p_k', [C_g/M_g])$ converge to asymptotically independent Poisson random variables with respect to the Hodge measure, and explain a connection with the random matrix equidistribution results of Katz-Sarnak \cite{katz-sarnak} in this setting as well as the settings of smooth hypersurface sections and configuration spaces. 

\subsubsection{Complete intersections}
Bucur and Kedlaya \cite{BucurKedlaya-Complete} have generalized Poonen's sieve results to complete intersections and used them to study the number of degree~1 points on a random complete intersection over a finite field. Using their results, one finds point-counting stabilization results analogous to Theorems~\ref{thm:StabCohomPointCount}~and~ \ref{thm:StableGeomPointCount} for the vanishing cohomology and relative configuration spaces of the universal smooth complete intersection. We expect that the techniques of Vakil and Wood, as extended in this paper to prove Theorems~ \ref{thm:StabCohomHodge}~and~\ref{thm:StabGeomMot}, can also be applied to complete intersections, but we do not carry this out in the present work. 

\subsubsection{Branched covers}
There has been some interesting related work on cyclic branched covers of $\PP^n$. The monodromy representations coming from these covers are studied by Carlson and Toledo \cite{CarlsonToledo-CyclicMonodromy} in their work on fundamental groups of discriminant complements. In the case $n=1$, Chen \cite{Chen-Burau} has recently proved a stability result for the cohomology with coefficients in the local system corresponding to the first cohomology of the universal family of cyclic branched covers. From this he deduces the corresponding point counting result, which he phrases using probability. It would be interesting to extend this computation to the natural local systems constructed out of the first cohomology via the $\lambda$-ring structure, and to extend the probabilistic interpretation to a motivic setting. 

For $n=1$, it would also be interesting to study these problems for Hurwitz spaces, where cohomological stabilization with trivial coefficients is shown in \cite{EVW-Hurwitz}.

\subsection{Outline}
In Section \ref{sec:ProbabilitySpaces} we introduce the algebraic probability measures (in the sense of \cite{Howe-MRV1}) used in our proofs. In Section \ref{sec:RepTheory} we describe some aspects of the representation theory of symmetric, symplectic, and orthogonal groups.  In Section~\ref{sec:vancohom} we recall the construction of vanishing cohomology and explain in more detail the construction of the local systems corresponding to a representation $\pi$. In Section \ref{sec:PointCountingResults} we prove our point counting Theorems \ref{thm:StabCohomPointCount} and \ref{thm:StableGeomPointCount}. Finally, in Section \ref{sec:BettiResults} we prove our stabilization results over $\CC$, Theorems \ref{thm:StabCohomHodge} and \ref{thm:StabGeomMot}.

\subsection{Notation}
For partitions and configuration spaces we follows the conventions of Vakil and Wood \cite{VakilWood-Discriminants}, except that where they would write $w_\tau$, we write $\Conf^\tau$. Also, we tend to avoid the use of $\lambda$ to signify a partition to avoid conflicts with the theory of pre-$\lambda$ rings. 

A variety over a field $\KK$ is a finite-type scheme over $\KK$. It is quasi-projective if it can be embedded as a locally closed subvariety of $\PP^n_\KK$.  

Our notation for pre-$\lambda$ rings and power structures is described in \cite[Section \ref{MRV1-sec:PreLambda}]{Howe-MRV1}. We highlight the following point here: if $f \in 1+(t_1,t_2,...)R[[t_1, t_2,..]]$, then 
$f^r$ will always denote the naive exponential power series
\[ \exp \left( r\cdot \log f \right) \in   1+(t_1,t_2,...)R_\QQ[[t_1, t_2,..]]. \]
If $R$ is a pre-$\lambda$ ring and we want to denote an exponential taken in the associated power structure, then we write it as
\[ f^{_\Pow r}. \]

Our notation for Grothendieck rings of varieties and motivic measures is explained in \cite[Section \ref{MRV1-sec:GrothendieckRing}]{Howe-MRV1}.

\subsection{Acknowledgements}
We thank Margaret Bilu, Bhargav Bhatt, Weiyan Chen, Matt Emerton, Benson Farb, Melanie Matchett Wood and Jesse Wolfson for helpful conversations. We thank Matt Emerton, Benson Farb, Melanie Matchett Wood, and an anonymous referee for helpful comments on earlier drafts of this paper. We thank the organizers, speakers, assistants, and participants of the 2014 and 2015 Arizona Winter Schools, where we first picked up many of the pieces of this puzzle~--~in particular, we thank Jordan Ellenberg, Bjorn Poonen, and Ravi Vakil for their courses, and Margaret Bilu for first pointing out to us the notion of a motivic Euler product.  

	This paper and its prequel \cite{Howe-MRV1} grew out of an effort to  systematically understand the explicit predictions for the cohomology of local systems on $U_d(\CC)$ that can be obtained from Poonen's sieving results \cite{Poonen-Bertini} (as in Theorem \ref{thm:StabCohomPointCount}). In particular, the results in the complex setting began as conjectures motivated by the finite field case, and it was only after reading Vakil and Wood's paper \cite{VakilWood-Discriminants} that we realized some predictions could be proved by working in the Grothendieck ring of varieties. We would like to acknowledge the intellectual debt this article owes to both \cite{Poonen-Bertini} and \cite{VakilWood-Discriminants}, which will be obvious to the reader familiar with these works. 

\section{Some probability spaces}\label{sec:ProbabilitySpaces}
In this section we define the algebraic probability spaces we use in the rest of the paper. For basic results on algebraic probability theory, we refer to \cite[Section~\ref{MRV1-sec:Probability}]{Howe-MRV1}. For convenience, here we recall that for a ring $R$, an $R$-probability measure $\mu$ on an $R$-algebra $A$ with values in an $R$-algebra $A'$ is given by an expectation map
\[ \EE_\mu : A \rightarrow A' \]
which is a map of $R$-modules sending $1_A$ to $1_{A'}$. Here we think of $A$ as being an algebra of random variables. 

In Subsection \ref{subsec:ProbLisse}, we discuss the spaces used in Section \ref{sec:PointCountingResults} to prove our finite field Theorems \ref{thm:StabCohomPointCount} and \ref{thm:StableGeomPointCount}. The main point is to work in a setting rich enough to handle cohomology groups of algebraic varieties, but small enough so that we can discuss convergence in the archimedean topology on $\QQ$ without too many gymnastics. We reach this middle ground by using the subring of the Grothendieck ring of lisse $l$-adic sheaves consisting of virtual sheaves with everywhere integral characteristic power series of Frobenius. 

In Subsection \ref{subsec:ProbLisse}, we discuss the spaces used to prove Theorem \ref{thm:StabCohomHodge} on the stabilization of the Hodge structures in the cohomology of the variations of Hodge structure $\VanQ^{\pi_{\sigma,d}}$. We note that the standard approach for producing the mixed Hodge structure on these cohomology groups is via Saito's \cite{Saito-MHM1,Saito-MHM2} theory of mixed Hodge modules, however, for our purposes it is simpler to use the geometric theory of Arapura \cite{Arapura-LerayMotivic}. Besides introducing less technical overhead, the main advantage of Arapura's theory for us is that compatibility with Deligne's \cite{Deligne-HodgeII} mixed Hodge structures and the Leray spectral sequence is built-in from the start. We use this compatibility when we deduce the cohomological stabilization in Theorem \ref{thm:StabCohomHodge} from the geometric stabilization in Theorem \ref{thm:StabGeomMot} (cf. Lemma \ref{lem:HodgeMotCommute} for the precise statement used). As indicated in the introduction of \cite{Arapura-LerayMotivic}, the mixed Hodge structures obtained via Saito's theory should agree with those used here. 

In Subsection \ref{subsec:MotRanVar} we discuss the \emph{motivic random variables} used in our motivic stabilization result, Theorem \ref{thm:StabGeomMot} (cf. also \cite[Section~\ref{MRV1-sec:GrothendieckRing}]{Howe-MRV1}). 

\subsection{Lisse $l$-adic sheaves} \label{subsec:ProbLisse}
For a variety $S/\FF_q$, denote by $\Loc_{\QQ_l} S$ the category of lisse $\QQ_l$-sheaves on $S$. The Grothendieck ring 
\[ K_0(\Loc_{\QQ_l}) \]
is a $\lambda$-ring with 
\[ \sigma_k([\Vc]) = [\Sym^k \Vc]. \]
Trace of Frobenius at a closed point $u\in S$ gives a map of rings
\begin{eqnarray*}
K_0(\Loc_{\QQ_l}) & \rightarrow & \QQ_l \\
\;K & \mapsto & \Tr \Frob_u \actsonr K_{\overline u}
\end{eqnarray*}
which sends 
\[ \sigma_t(K) = 1 + \sigma_1(K)t + \sigma_2(K)t^2 + \ldots \]
to the characteristic power series of Frobenius acting on $K_{\overline{u}}$. 
We denote by 
\[ K_0(\Loc_{\QQ_l} S)' \] 
the sub-$\lambda$-ring consisting of elements $K$ such that this character power series is in $\ZZ[[t]]$ for all closed points. By \cite[Theorem 1.6]{Deligne-WeilI} and smooth proper base change, if $f: T \rightarrow S$ is smooth and proper, then for any $i$,
\[ [R^i f_* \QQ_l] \in K_0(\Loc_{\QQ_l} S)'. \]
For any smooth proper $f:T \rightarrow S$, we will denote 
\[ [T]:= [Rf_* \QQ_l] = \sum_i (-1)^i [R^i f_* \QQ_l] \in K_0(\Loc_{\QQ_l} S)'. \]

Pullback via $S \rightarrow \Spec \FF_q$ induces a map of $\lambda$-rings 
\[ K_0(\Loc_{\QQ_l} \FF_q)' \rightarrow  K_0(\Loc_{\QQ_l} S)' \]
with image the constant sheaves. 

In the other direction, there is a map of $K_0(\Loc_{\QQ_l} \FF_q)'$-modules 
\begin{eqnarray*}
K_0(\Loc_{\QQ_l} S)' & \xrightarrow{\chi_S} & K_0(\Loc_{\QQ_l} \FF_q)' \\
\; [\Vc] & \mapsto & \sum_i (-1)^i [ H^i_c (S_{\overline{\FF_q}}, \Vc ) ]. 
\end{eqnarray*}

By the Grothendieck-Lefschetz fixed point formula, the composition of $\chi_S$ with $\Tr \Frob_q$ to $\ZZ$ is given by 
\[ K \mapsto \sum_{u \in S(\FF_q)} \Tr \Frob_q \actsonr K_{\overline{u}} \]
(in particular, this verifies that the image of $\chi_S$ is actually in $K_0(\Loc_{\QQ_l} \FF_q)'$). The image of $\underline{\QQ_l}$ under this composed map is $\#S(\FF_q)$.

In particular, if we consider $\QQ$ as an algebra over $K_0(\Loc_{\QQ_l} \FF_q)'$ via $\Tr \Frob_q$, we obtain a $\QQ$-valued probability measure $\mu$ on $K_0(\Loc_{\QQ_l} / S)'$ via 
\[ \EE_\mu : K \mapsto \frac{\Tr \Frob_q \actsonr \chi_S(K)}{\#S(\FF_q)}. \]

If we fix a $K \in K_0(\Loc_{\QQ_l} S)'$, then we obtain a map of $K_0(\Loc_{\QQ_l} \FF_q)'$ algebras 
\begin{eqnarray*}
 \Lambda_{K_0(\Loc_{\QQ_l} \FF_q)'} & \rightarrow & K_0(\Loc_{\QQ_l} S)' \\
 g & \mapsto & (g, K) 
\end{eqnarray*}
and, by composition with $\EE_\mu$, a $\QQ$-valued $K_0(\Loc_{\QQ_l} \FF_q)'$ measure on $\Lambda_{K_0(\Loc_{\QQ_l} \FF_q)'}.$

\newcommand{\GVHS}{\mathrm{GVSH}}
\newcommand{\GHS}{\mathrm{GHS}}

\subsection{Variations of Hodge structure}\label{subsec:ProbVHS}
Let $S/\CC$ be a smooth variety. Following Arapura \cite{Arapura-LerayMotivic}, we denote by $\GVHS(S)$ the category of \emph{geometric} polarizable variations of $\QQ$-Hodge structure on $S(\CC)$ -- it is the full subcategory of variations of Hodge structures consisting of those isomorphic to a direct summand of $R^i f_* \QQ$ for a smooth projective $f: T\rightarrow S$. 

By \cite{Arapura-LerayMotivic}, Theorem 5.1, for $\Vc \in \GVHS(S)$, $H^i(S, \Vc)$ is equipped with a natural mixed Hodge structure. It is compatible with Leray spectral sequences for smooth proper maps and the Deligne \cite{Deligne-HodgeII} mixed Hodge structures on the cohomology of smooth varieties. Indeed, the results of \cite{Arapura-LerayMotivic} show that the Leray filtration is a filtration by mixed Hodge structures, and the mixed Hodge structure on $H^i(S, \Vc)$ is \emph{defined} to be the induced mixed Hodge structure for any smooth proper $f$ such that $\Vc$ is a direct summand of $R^i f_* \QQ$. 

For $\Vc \in \GVHS(S)$ of weight $w$, we define the Hodge structure on compactly supported cohomology using Poincar\'{e} duality
\[ H^i_c(S(\CC), \Vc):= H^{2n-i}(S, \Vc)^*(-w - \dim S). \]

We denote by
$K_0(\GVHS(S))$ the corresponding Grothendieck ring, which is a $\lambda$-ring with 
\[ \sigma_k([\Vc])=[\Sym^k \Vc]. \]
That $\Sym^k \Vc$ is also geometric for $\Vc$ geometric follows by finding it as a direct summand of the relative cohomology of a product, as in the proof of Corollary 5.5 of \cite{Arapura-LerayMotivic}. 

For $S=\CC$, this is the Grothendieck ring of geometric polarizable $\QQ$-Hodge structures $K_0(\GHS)$. It is a subring of the Grothendieck ring of polarizable Hodge structures $K_0(\HS)$, as defined in \ref{subsec:IntroCohomStab}. 

For any $T/S$ smooth and projective, $R^i f_* T$ is naturally a variation of Hodge structure, and thus gives a class 
\[ [R^i f_* T] \in K_0(\GVHS(S)). \]
In this setting, we denote 
\[ [T/S]_\GVHS := \sum_i (-1)^i [R^i f_* T] \in K_0(\GVHS( S)) \] 
(here the subscript is to distinguish from the class $[T/S]$ in $K_0(\Var S)$). 

Pullback via $S \rightarrow \Spec \CC$ gives a map of $\lambda$-rings 
\[ K_0 (\GHS) \rightarrow K_0(\GVHS(S)) \]
with image the constant geometric variations of Hodge structure on $S$. 

In the other direction, there is a map of $K_0(\GHS)$-modules 
\begin{eqnarray*}
K_0 (\GVHS_\QQ S) & \xrightarrow{\chi_S} & K_0(\HS) \\
\;[\Vc] & \mapsto & \sum_{i} (-1)^i [H^i_c (S(\CC), \Vc)] 
\end{eqnarray*}

\begin{remark}
The map $\chi_S$ factors through $K_0(\GHS) \subset K_0(\HS)$, but we will not make use of this fact. 
\end{remark}

Given a smooth projective $T/S$, we will later want to deduce stabilization results for the variations of Hodge structures constructed from $T$ from motivic stabilization results for $T$. To do so, we need a lemma that says certain pre-$\lambda$ ring structures are compatible:

\begin{lemma}\label{lem:HodgeMotCommute}
For $T/S$ smooth and projective, the following diagram commutes. 
\[
\begin{tikzcd} 
 &  K_0 (\Var / S) \ar[r, "\mathrm{forget}"] & K_0(\Var) \ar[dr, "\chi_\HS"] \\
\Lambda \ar[ur, "{f \mapsto (f, [T/S])}"] \ar[dr,swap, "{f \mapsto (f,[T{/}S]_\GVHS)}"] & & &  K_0(\HS) \\
& K_0(\GVHS(S)) \ar[rru, "\chi_S"]
\end{tikzcd}  
\]
\end{lemma}
\begin{proof}
 
It suffices to verify that for any $l_1,...,l_m$, the two maps from $\Lambda$ to $K_0(\HS)$ agree on $h_{l_1}\cdot...\cdot h_{l_m}$. We denote $n=\sum l_i$, and $S_{\overline{l}} = S_{l_1} \times ... \times S_{l_m}$. 

Via the top arrows, we obtain the cohomology of the quotient
\[ T^{\times_S n} / S_{\overline{l}}. \]
Here the superscript on $T$ denotes the $n$-fold fiber product over $S$. 

This is equal to the $S_{\overline{l}}$-invariants of the cohomology of $T^{\times_S n}$, which is computed by the $S_{\overline{l}}$-invariants of the Leray spectral sequence for 
\[  \pi_n: T^{\times_S n} \rightarrow S. \]

From the $E_2$ page we obtain
\[ \chi_\HS (T^ {\times_S n} / S_{\overline{l}}) = \sum_{i,j} (-1)^{i+j} [H^j_c(S, (R^i \pi_{n*} \QQ)^{S_{\overline{l}}}) ]= \chi_S \left( \sum_i (-1)^i [(R^i \pi_{n *} \QQ) ^{S_{\overline{l}}} ] \right).\]

It suffices then to show
\[ \sum_i (-1)^i [(R^i \pi_{n*} \QQ) ^{S_{\overline{l}}} ] = (h_{l_1}\cdot...\cdot h_{l_m}, [T/S]_\GVHS). \]
 
\newcommand{\bigdsum}{\bigoplus}
We can compute $R^i \pi_{n*} \QQ$ explicitly as an $S_{\overline{l}}$-equivariant variation of Hodge structure using the relative Kunneth formula: Let
\[ V_{n,\textrm{odd}}=\bigdsum_{i \textrm{ odd}} R^i \pi_{n*} \QQ \]
and 
\[ V_{n,\textrm{even}}=\bigdsum_{i \textrm{ even}} R^i \pi_{n*} \QQ. \]
Similarly, for 
\[ \pi: T \rightarrow S, \]
let
 \[ V_{\textrm{odd}}=\bigdsum_{i \textrm{ odd}} R^i \pi_{*} \QQ \]
and 
\[ V_{\textrm{even}}=\bigdsum_{i \textrm{ even}} R^i \pi_{*} \QQ. \]
Then,
\newcommand{\Ind}{\mathrm{Ind}}
\[ V_{n,\textrm{odd}} = \bigdsum_{k_j \leq l_j, \sum k_j \textrm{ odd} } \; \bigboxtimes_{j=1}^m \Ind_{S_{k_j}\times S_{l_j-k_j}}^{S_{l_j}} \left( \left( V_{\textrm{odd}}^{k_j} \otimes \mathrm{sgn} \right) \dsum V_{\textrm{even}}^{l_j-k_j} \right),  \]
and
\[ V_{n,\textrm{even}} = \bigdsum_{k_j \leq l_j, \sum k_j \textrm{ even} } \; \bigboxtimes_{j=1}^m \Ind_{S_{k_j}\times S_{l_j-k_j}}^{S_{l_j}} \left( \left( V_{\textrm{odd}}^{k_j} \otimes \mathrm{sgn} \right) \dsum V_{\textrm{even}}^{l_j-k_j} \right).  \]
Applying Frobenius reciprocity, we find
\[ V_{n,\textrm{odd}}^{S_{\overline{l}}} =  \bigdsum_{k_j \leq l_j, \sum k_j \textrm{odd} } \; \bigotimes_{j=1}^{m} \left( \extp^{k_j} V_{\textrm{odd}} \otimes \Sym^{l_j - k_j} V_{\textrm{even}} \right) ,\]
and 
\[ V_{n,\textrm{even}}^{S_{\overline{l}}} =  \bigdsum_{k_j \leq l_j, \sum k_j \textrm{even} } \; \bigotimes_{j=1}^{m} \left( \extp^{k_j} V_{\textrm{odd}} \otimes \Sym^{l_j - k_j} V_{\textrm{even}} \right).\]
Thus, 
\begin{multline}\label{eqn:OneWayLeray}\sum_i (-1)^i [(R^i \pi_{n*} \QQ) ^{S_{\overline{l}}} ] =  [V_{n,\textrm{even}}^{S_{\overline{l}}}]-[V_{n,\textrm{odd}}^{S_{\overline{l}}}] \\= \prod_{j=1}^{m} \sum_{k_j \leq l_j} (-1)^{k_j} [\extp^{k_j} V_{\textrm{odd}}]\cdot [\Sym^{l_j - k_j} V_{\textrm{even}}]. \end{multline}

We claim that 
\[ \sum_{k_j \leq l_j} (-1)^{k_j} [\extp^{k_j} V_{\textrm{odd}}]\cdot [\Sym^{l_j - k_j} V_{\textrm{even}}] = (h_{l_j}, [T/S]_\GVHS). \]
Given this claim, (\ref{eqn:OneWayLeray}) finishes the proof. To see this claim, we observe
\begin{multline*}
 (h_{l_j}, [T/S]_\GVHS)  = (h_{l_j}, [V_{\textrm{even}}]-[V_{\textrm{odd}}]) \\ = \sum_{k_j \leq l_j} (-1)^{k_j}(h_{l_j-k_j}, [V_{\textrm{even}}])\cdot(e_{k_j},[V_{\textrm{odd}}])   \\
= \sum_{k_j \leq l_j}  (-1)^{k_j} [\Sym^{l_j-k_j}V_{\textrm{even}}] \cdot [\extp^{k_j}V_{\textrm{odd}}].
\end{multline*}
Here the second line comes from identifying $(h_{l_j}, [V_{\textrm{even}}]-[V_{\textrm{odd}}])$ with the coefficient of $t^{l_j}$ in 
\[ \sigma_t([V_{\textrm{even}}]-[V_{\textrm{odd}}])=\sigma_t([V_{\textrm{even}}])\sigma_{t}([V_{\textrm{odd}}])^{-1}, \]
and computing the term $\sigma_{t}([V_{\textrm{odd}}])^{-1}$ using the symmetric function identity
\[ (1+h_1 t + h_2 t^2 + h_3 t^3 + ..)^{-1} = 1 - e_1 t + e_2 t^2 - e_3 t^3 +... \]
to identify the coefficient of $t^k$ with $(-1)^k(e_k, [V_{\textrm{odd}}])=(-1)^k[\extp^k V_{\textrm{odd}}].$ 
\end{proof}

Finally, since 
\[ [S]_\HS=\chi_S ([S/S]_\GVHS)=\chi_S(\underline{\QQ}) \]
is invertible in $\widehat{K_0(\HS)}$, we obtain a $K_0(\GHS)$-probability measure on $K_0(\GVHS(S))$ with values in $\widehat{K_0(\HS)}$ by
\begin{eqnarray*}
 \EE_\mu: K_0(\GVHS(S)) & \rightarrow & K_0(\HS) \\
 K & \mapsto & \frac{\chi_S(K)}{[S]_\HS}. 
\end{eqnarray*}

\subsection{Motivic random variables}\label{subsec:MotRanVar}
If $S/\KK$ is a variety, we consider the relative Grothendieck ring of varieties $K_0(\Var / S)$, which we view as the ring of \emph{motivic random variables} on $S$. It is an algebra over $K_0(\Var)$, the Grothendieck ring of varieties over $\KK$. If 
\[ \phi: K_0(\Var) \rightarrow R \]
is a motivic measure such that $[S]_\phi$ is invertible in $R$, we obtain an $R$-valued $K_0(\Var)$-probability measure on $K_0(\Var / S)$ by 
\[ \EE_\mu [ [Y/S] ] = \frac{[Y]_\phi}{[S]_\phi}. \]
In this article, the most important example of $\phi$ will be the natural map 
\[ K_0(\Var) \rightarrow \motcomp. \]
When $\KK$ is of characteristic zero, the relative Kapranov zeta function gives $K_0(\Var / S)$ the structure of a pre-$\lambda$ ring, and given $Y/S$, we can pull back $\EE_\mu$ via
\[ f \mapsto (f, [Y/S]) \]
to obtain an $R$-valued $\ZZ$-probability measure on $\Lambda$. 

For more details on the relative Grothendieck ring $K_0(\Var /S)$ and the pre-$\lambda$ structure, cf. \cite[Section \ref{MRV1-sec:GrothendieckRing}]{Howe-MRV1}. 

\section{Some representation theory}\label{sec:RepTheory}
In this section we explain how to parameterize irreducible representations of symmetric, orthogonal, and symplectic groups by partitions $\sigma$ in order to produce the representations $\pi_{\sigma,d}$ described in Subsection \ref{subsec:IntroCohomStab}. The key point in common between these groups is that each such group $G$ admits a standard representation $V_\std$ inducing via the pre-$\lambda$ structure a surjection
\begin{eqnarray*} 
\Lambda & \twoheadrightarrow & K_0(\Rep G)\\
f & \mapsto & (f, [V_\std]).
\end{eqnarray*}
In other words, the representation ring is spanned by products of symmetric powers of the standard representation. Through these surjections, families of irreducible representations are naturally parameterized by certain Schur polynomials. For orthogonal and symplectic groups, the results in this section follow from work of Koike and Terada \cite{KoikeTeradaYoungMethods},  and for symmetric groups from work of Marin \cite{Marin-HooksGenerate}.

Let $G$ be a linear algebraic group over a field $F$ of characteristic zero, and consider the category $\Rep G$ of algebraic representations of $G$ on finite dimensional vector spaces over $F$. For us, the most important examples are $F=\QQ$ or $\QQ_l$ and 
\newcommand{\Sp}{\mathrm{Sp}}
\newcommand{\Orth}{\mathrm{O}}
\newcommand{\SO}{\mathrm{SO}}
\begin{equation}\label{eqn:GroupsUsed} G = \begin{cases} S_N \textrm{, a symmetric group } \\ \Sp(V, \langle\;,\;\rangle) \textrm{, automorphisms of a non-degenerate symplectic form} \\
\Orth(V, \langle\;,\;\rangle) \textrm{, automorphisms of a non-degenerate symmetric form} \end{cases}
\end{equation}
In the latter two cases we will also consider the corresponding group of homotheties~$G'$. 

In each of the cases $\bullet=\{S, \Sp, \Orth\}$ of Equation (\ref{eqn:GroupsUsed}), we have a way to assign to a partition $\sigma=(\sigma_1,...,\sigma_k)$ a family of irreducible representations $V_{G,\sigma}$ of the groups in $\bullet$ of sufficient size depending on $\sigma$ (here size means $N$ for symmetric groups and the dimension of $V$ for symplectic and orthogonal groups):
\begin{enumerate}
\item For symmetric groups, we use the method of \cite[Section 2.1]{ChurchFarb-RepStab}. The irreducible representation of $S_N$ attached to $\sigma$ is obtained via the theory of Young tableaux from the partition
\[ (N-|\sigma|, \sigma_1,...,\sigma_k). \]
In particular, it is defined only for $|N| \geq |\sigma| + \sigma_1$. 
\item For symplectic groups (all of which are split over $F$), we use highest weight theory as in \cite[ Section 1]{KoikeTeradaYoungMethods} (see also \cite[Section 2.2]{ChurchFarb-RepStab}). The partition $\sigma$ defines a dominant integral weight of $\Sp_{k+m}$ by padding with $m$ zeros at the end, and from this we obtain an irreducible representation. 
\item For orthogonal groups, there are two complications: first, the group is not connected (the components correspond to $\det=\pm 1$), and second, we will typically be discussing orthogonal groups which are not split over $F$. Nevertheless, after padding the right by zeroes, the partition $\sigma$ will give rise as in \cite[ Section 1]{KoikeTeradaYoungMethods} to a representation of $\SO(V_{\overline{F}})$ obtained as the restriction of an irreducible representation of $\Orth(V_{\overline{F}})$, and we will see below in Remark \ref{rmk:LiftRep} that there is a natural way to choose a representation of $\Orth(V)$ (and even the group of homotheties) restricting to this representation and defined over~$F$. 
\end{enumerate} 

For $G$ as in Equation (\ref{eqn:GroupsUsed}), we denote by $V_\std$ the standard representation -- when $G=S_N$ it is the space of functions on $\{1,...,N \}$ summing to zero\footnote{It is sometimes more natural to take the full permutation representation on $\{1,...,N\}$ as the standard representation in this case.} and when $G$ is the automorphism group of a pairing on $V$ it is given by $V$ itself. 

The Grothendieck ring $K_0(\Rep G)$ is equipped with a pre-$\lambda$ ring structure by
\[ \sigma_k([V])= [\Sym^k V]. \] 
It has the following useful interpretation: if we identify a virtual representation $K$ with its trace in $F[G]^G$ (action by conjugation), then for $f \in \Lambda$, $(f, [V])$ is identified with the function which sends an element $g \in G(\overline{F})$ to the symmetric function $f$ evaluated on the eigenvalues of $g$ acting on $V$. 

\begin{proposition}\label{prop:SurjLambdaSchur}
For $G$ as in Equation (\ref{eqn:GroupsUsed}), the map
\begin{eqnarray*}
\Lambda & \rightarrow & K_0(\Rep G)\\
f & \mapsto & (f, [V_\std]) 
\end{eqnarray*}
is surjective. Furthermore, if we fix one of the three cases $\bullet\in\{S,\Sp, \Orth \}$ of Equation~(\ref{eqn:GroupsUsed}), then for any partition $\sigma$, there is a \emph{Schur polynomial} $s_{\bullet,\sigma}$ such that 
\[ (s_{\bullet,\sigma}, [V_\std]) = V_{\sigma} \]
for any $G$ (of sufficient size) of the form $\bullet$. 
\end{proposition}

\begin{remark} \label{rmk:LiftRep} In the case of orthogonal groups, we had only defined $V_\sigma$ up to restriction to $\SO(V_{\overline F})$. The orthogonal Schur polynomials of Proposition \ref{prop:SurjLambdaSchur} pick out a lift to 
$\Orth(V_{\overline F})$, defined over $F$. 

	Furthermore, for both symplectic and orthogonal groups, $V_\std$ is naturally a representation of the larger homothety group $G'$ (i.e., the group preserving the pairing up to a scalar multiple instead of on the nose), and thus Proposition \ref{prop:SurjLambdaSchur} picks out a natural lift of $V_\sigma$ to a representation of $G'$ by
\[ [V_\sigma] = (s_{\bullet,\sigma}, [V_\std]) \in K_0(\Rep G'). \]
\end{remark}

\begin{proof}
We may assume $F=\overline{F}$, as for any reductive group $G/F$, the base extension map $K_0(\Rep G) \rightarrow K_0(\Rep G_{\overline{F}})$ is an injection\footnote{We remind the reader of why this holds: the Grothendieck ring is the free abelian group on the isomorphism classes of irreducible representations, so it suffices to show that any two distinct irreducible representations of $G$ over $F$ do not share any irreducible summands after base extension to $\overline{F}$. This is true because, for $V_1$ and $V_2$ any two representations of $G$, \[ \Hom_{G_{\overline{F}}}(V_{1, \overline{F}}, V_{2,\overline{F}}) = \Hom_{G}(V_1,V_2)\otimes_F \overline{F}. \]}.

We first handle the symplectic and orthogonal cases. The surjectivity in the symplectic case then follows from \cite[Proposition 1.2.6]{KoikeTeradaYoungMethods}, noting that the generator $p_i$ there corresponds to 
$[\Sym^i V_\std] =(h_i, [V_\std])$ in our notation. In the orthogonal case, we caution the reader that the ring $R(\Orth(m))$ of \cite{KoikeTeradaYoungMethods} is the \emph{image} of  $K_0(\Rep(\Orth(m)))$ in $K_0(\Rep \SO(m))$. Surjectivity onto this ring follows from the same proposition, noting that the generator $e_i$ there corresponds to 
\[ [\extp^i V_\std] = (e_i, [V_\std]) \]
in our notation. To obtain surjectivity onto $K_0(\Rep(\Orth(m))$, we observe that two irreducible representations of $\Orth(m)$ with equal restriction to $\SO(m)$ differ by tensoring with $\det$, and thus if one is in the image of $\Lambda$, we obtain the other via multiplication by $e_m$. We define the Schur polynomials using the formulas of \cite[Theorems 1.3.2 and 1.3.3]{KoikeTeradaYoungMethods}, together with the observation of \cite[Remark 1.3.4]{KoikeTeradaYoungMethods} that for the dimension large relative to $\sigma$, the formulas are independent of the group.  

For symmetric groups, surjectivity and the existence of Schur polynomials follows from the proof of \cite[Proposition 4.1]{Marin-HooksGenerate}, which proves surjectivity essentially by inductively constructing such Schur polynomials from the elementary symmetric polynomials $e_k$. 
\end{proof}

\begin{remark}
For symmetric groups, we observe that if $X_i$ is the function counting cycles of length $i$, then the character of $\extp^k (V_\std \dsum 1)$ is given by
\[ (-1)^k \cdot \sum_{(\tau_1,...,\tau_l) \textrm{ s.t. } \sum i \tau_i = k} (-1)^{\sum \tau_i} \cdot  \binom{X}{\tau}  \]
where
\[ \binom{X}{\tau} = \prod_i \binom{X_i}{\tau_i}. \]
Combining this with \cite[Section 1.7, Example 14]{Macdonald-SymmetricFunctions}, one can obtain an explicit formula for the Schur polynomials as polynomials in the $X_i$ (which correspond, as in \cite{Howe-MRV1}, to Mobius inverted power sum polynomials $p_i'$ in $\Lambda_\QQ$).  

It would be interesting to compare further the two integral structures on character polynomials occurring naturally here -- the first coming from $\Lambda$, the second coming from the span of the
\[ \binom{X}{\tau}, \]
which give rise to integer valued class functions. 
\end{remark}

\section{Vanishing cohomology}\label{sec:vancohom}
In this section we recall the construction of the local systems $\VanQ$ and $\VanQl$ of vanishing cohomology and establish some of the related notation used in the rest of this paper. 

Let $Y$ be a polarized smooth projective variety of dimension $n$ over a field $\KK$, with polarization denoted by $\Lc$. For any $d \in \ZZ_{\geq1}$, we consider the affine space of global sections of $\Lc^d$, $\AA(\Gamma(Y, \Lc^{\otimes d}))$, and inside of it the Zariski open space of sections with smooth vanishing locus, 
\[  U_{d} \subset  \AA(\Gamma(Y, \Lc^{\otimes d})). \] 

There is a universal smooth projective hypersurface section
\[ f: Z_d \rightarrow U_d \]
fitting into a commutative diagram 
\[ \begin{tikzcd} Z_d \arrow[dr, "f"] \arrow[r, hook] & Y \times U_d \arrow[d]\\
& U_d \end{tikzcd} \] 
such that for $u \in U_d(\overline{\KK})$, the fiber 
\[ Z_{d,u} \hookrightarrow Y \]
is the smooth closed subvariety of $Y_{\overline{\KK}}$ corresponding to $u$. 

\begin{definition}
Let $l$ be coprime to the characteristic of $\KK$. The \emph{\'{e}tale vanishing cohomology} $\VanQl$ is the lisse $\QQ_l$-sheaf on $U_d$ 
\[ \VanQl := \ker R^{n-1} f_* \QQ_l \rightarrow \underline{H^{n+1}(Y_{\overline{\KK}}, \QQ_l)(1)} \]
where the underline denotes the constant sheaf and the map is the relative Gysin map. 
If $\KK=\CC$, the \emph{Betti vanising cohomology} 
$\VanQ$ is the $\QQ$-local system on $U_d(\CC)$ 
\[ \VanQ := \ker R^{n-1} f_{\an*} \QQ \rightarrow \underline{H^{n+1}(Y(\CC), \QQ)(1)} \]
where $f_{\an}$ denotes the analytification of $f$. 
\end{definition}  

\begin{remark}
Geometrically, the Gysin map is given fiberwise by the natural map of homology 
\[ H_{n-1}(Z_u) \rightarrow H_{n-1}(Y) \]
after writing homology as dual to cohomology and then using Poincar\'{e} duality. The kernel is generated by the classes of vanishing spheres in a Lefschetz pencil through $Z_u$, thus the name vanishing cohomology. For more on this geometric interpretation, cf. e.g. \cite[Section 2.3.3]{Voisin-HTII}. 
\end{remark}

\begin{example}
If $Y=\PP^n$, then $U_d$ parameterizes smooth hypersurfaces of degree $d$, and $\VanQl$ (resp $\VanQ$) is the primitive part; in particular, if $n=\dim Y$ is odd then the vanishing cohomology is equal to $R^{n-1}f_{*}\QQ_l$  (resp. $R^{n-1}f_{\an *}\QQ$). The latter holds also for any $Y$ of odd dimension that is a complete intersection in $\PP^m$. 
\end{example}

The local system $R^{n-1} f_* \QQ_l$ (resp $R^{n-1} f_{\an*} \QQ$) is equipped with a non-degenerate pairing to $\QQ_l(-(n-1))$ (resp $\QQ(-(n-1))$) whose restriction to $\VanQl$ (resp $\VanQ$) is also non-degenerate; we denote this restricted pairing by $\langle\;,\;\rangle$. It is symmetric if $n-1$ is even and anti-symmetric if $n-1$ is odd. 

By \cite[Corollaire 4.3.9 ]{Deligne-WeilII} (resp. by the hard Lefschetz theorem over $\CC$), we have
\[ \VanQl^\perp \isoeq \underline{H^{n-1}(Y_{\overline{\KK}}, \QQ_l)} \]
(resp. 
\[ \VanQ^\perp \isoeq \underline{H^{n-1}(Y(\CC), \QQ)}), \]
and there is a direct sum decomposition 
\begin{equation}\label{eqn:VanCohomDSum-ladic}
 R^{n-1} f_* \QQ_l \isoeq \VanQl \oplus \underline{H^{n-1}(Y_{\overline{\KK}}, \QQ_l)} 
\end{equation}
(resp. 
\begin{equation}\label{eqn:VanCohomDSum-Hodge} R^{n-1} f_{\an*} \QQ \isoeq \VanQ \oplus \underline{H^{n-1}(Y(\CC), \QQ)}). 
\end{equation}

If we fix a base point $u_0 \in U(\overline{\KK})$ and a trivialization of $\QQ_l(1)|_{u_0}$, the local system $\VanQl$ (resp. $\VanQ$) is determined by the monodromy representation 
\[ \rho_l: \pi_{1,\et}(U_d, u_0) \rightarrow \Aut'( \Vc_{\van,\QQ_l, u_0}, \langle\;,\;\rangle) \]
(resp. 
\[ \rho: \pi_{1}(U_d(\CC), u_0) \rightarrow \Aut( \Vc_{\van,\QQ, u_0}, \langle\;,\;\rangle ) ). \]
where by $\Aut'$ we denote the group of homotheties of the pairing.

The local system $\VanQl$ (resp. $\VanQ$) or, equivalently, the corresponding monodromy representation $\rho_l$ (resp. $\rho$) is absolutely irreducible. In fact, we can say much more: by Deligne \cite[ Th\'{e}or\`{e}mes 4.4.1 and 4.4.9]{Deligne-WeilII}, the image of $\pi_{1,\et}(U_{d, \overline{K}})$ under $\rho_l$ is either open or finite and equal to the Weyl reflection group of a root system of type $A$, $D$, or $E$ embedded in 
\[ (\Vc_{\van,\QQ_l, u_0}, \langle\;,\;\rangle ) \] 
(the vanishing cycles). The latter can occur only when $n-1$ is even, so that the pairing is symmetric, and, by a result of Katz \cite{Katz-LarsenAlt}, only for small $d$. Open image comes as a consequence of Zariski density, and the argument that the image is either Zariski dense or a reflection group as above given in \cite{Deligne-WeilII} is valid also in the topological setting for $\rho$ (c.f., e.g.,  \cite[Section 3.2]{Voisin-HTII}, for the vanishing cycles input). For the symmetric group case, \cite[2.4.4]{Katz-LarsenAlt}, e.g., shows the monodromy is surjective. 

Let $G_{\QQ_l}$ denote the algebraic group over $\QQ_l$ (resp. over $\QQ$) 
\[ G_{\QQ_l} := \Aut( \Vc_{\van, \QQ_l, u_0}, \langle\;,\;\rangle) \]
(resp. 
\[ G := \Aut( \Vc_{\van, \QQ, u_0}, \langle\;,\;\rangle) ). \]
Denote by $G'_{\QQ_l}$ (resp. $G'$) the corresponding group of homotheties.

Let $\Rep G_{\QQ_l}$ (resp. $\Rep G$) denote the category of algebraic representations of $G_{\QQ_l}$ on $\QQ_l$-vector spaces  (resp. $\QQ$-vector spaces) and by $\Loc_{\QQ_l} U_d$ (resp. $\Loc_\QQ U_d(\CC)$) the category of lisse $\QQ_l$-local systems on $U_d$ (resp. $\QQ$-local systems on $U_d(\CC)$). We obtain a functor 
\begin{eqnarray*} 
\Rep G_{\QQ_l} & \rightarrow & \Loc_{\QQ_l} U_{d, \overline{K}} \\
\pi & \rightarrow & (\VanQl|_{U_{d, \overline{K}}})^\pi
\end{eqnarray*}
where $(\VanQl|_{U_{d, \overline{K}}})^\pi$ is the local system corresponding to the representation 
\[ \pi \circ \rho_l|_{\pi_{1,\et}(U_{d,\overline{K}}, \overline{u_0})}. \]

If $\KK$ does not contain the $l$-power roots of unity, then in general the image of the monodromy representation is contained in $G'_{\QQ_l}$ but not in $G_{\QQ_l}$ (unless $n=0$). Thus we obtain a functor 
\begin{eqnarray*}
\Rep G'_{\QQ_l} & \rightarrow & \Loc_{\QQ_l}{U_d} \\
\pi' & \rightarrow & \VanQl^{\pi'}
\end{eqnarray*}
fitting into a commutative diagram 
\[
\begin{tikzcd} 
\Rep G'_{\QQ_l} \ar[rr, "\pi' \mapsto \VanQl^{\pi'}"]\ar[d, swap, "\pi' \mapsto \pi'|_{G_{\QQ_l}} "] & & \Loc_{\QQ_l} {U_d} \ar[d, "\Vc \mapsto \Vc|_{U_{d,\overline{\KK}}}"] \\
\Rep G_{\QQ_l} \ar[rr,"\pi \mapsto (\VanQl|_{U_{d,\overline{\KK}}})^{\pi} "] & & \Loc_{\QQ_l} {U_{d,\overline{\KK}}}. 
\end{tikzcd}
\]

Similarly, if $\KK=\CC$, then $\VanQ$ is equipped with a natural polarized variation of $\QQ$-Hodge structure. The map from $\Rep G$ to local systems is enriched to a map
\[ \Rep G' \rightarrow \VHS_{U_d(\CC)} \]
where the right-hand side denotes the category of polarizable variations of Hodge structure on $U_d(\CC)$.
It fits into a commutative diagram
\[
\begin{tikzcd}
\Rep G' \ar[rr, "\pi' \mapsto \VanQ^{\pi'}"] \ar[d, swap, "\pi' \mapsto \pi'|_G"] & & \VHS_\QQ U_d(\CC) \ar[d, "\textrm{forget Hodge filtration}"] \\
\Rep G \ar[rr, "\pi \mapsto \VanQ^{\pi}"] & & \Loc_\QQ U_d(\CC). 
\end{tikzcd} 
\]

\section{Point counting results}\label{sec:PointCountingResults}
In this section we prove Theorems \ref{thm:StabCohomPointCount} and \ref{thm:StableGeomPointCount}. As noted in Remark \ref{rmk:PoonenAsympInd}, after setting up the necessary language Theorem \ref{thm:StableGeomPointCount} is a simple consequence of results of Poonen \cite{Poonen-Bertini}. Our main contributions here are the reinterpretation of Poonen's results in the language of asymptotic independence and the pre-$\lambda$ ring structure, and from this the deduction of Theorem \ref{thm:StabCohomPointCount}.

We use the setup of Section \ref{sec:vancohom} with $\KK=\FF_q$. 

\subsection{Geometric stabilization}
In this subsection we prove Theorem \ref{thm:StableGeomPointCount}. 
 
We consider the classical discrete uniform probability measures 
\begin{eqnarray*} \Map(U_d(\FF_q), \QQ) & \xrightarrow{\EE_{\mu_d}} & \QQ \\
f & \mapsto & \frac{1}{\#U_d(\FF_q)} \sum_{u \in U_d(\FF_q)} f(u)
\end{eqnarray*}

Denote by $A(Y)$ the free polynomial ring over $\QQ$ with generators $X_y$ for $y$ running over the closed points of $Y$.

For each $d$, we obtain a map 
\[ A(Y) \rightarrow \Map(U_d(\FF_q), \QQ) \]
given by sending a closed point $X_y$ to the indicator random variable $X_y$ on $U_d(\FF_q)$ defined by
\[ X_y(u) = \begin{cases} 1 \textrm{ if y }\in Z_{d,u}  \\
						0 \textrm{ if y }\not\in Z_{d,u}. \end{cases} \]

By composition, we can think of $\mu_d$ as a sequence of $\QQ$-probability measures on $A(Y)$. 

\begin{theorem}[Poonen]\label{theorem:ClosedPointsInd} The random variables in $A(Y)$ 
\[ \{ X_y \}_{y \in Y \textrm{ a closed point} } \]
are asymptotically independent with asymptotic Bernoulli distributions 
\[ \EE_{\mu_\infty}[(1+t)^{X_y}] = 1 + \frac{q^{n \cdot \deg y} -1}{q^{(n+1) \cdot \deg y}-1} \cdot t. \]
\end{theorem} 

In particular, because the $X_y$ generate $A(Y)$ as a $\QQ$-algebra, the asymptotic measure $\mu_\infty$ is defined on $A(Y)$ and, for any given $a \in A(Y)$, $\EE_{\mu_\infty}[a]$ can be computed explicitly by writing it as a sum of monomials in the $X_y$. 

There are also natural point counting random variables in $\Map(U_d(\FF_q), \QQ)$ coming from the universal family:
\[ X_k (u) = \# \textrm{ closed points of degree } k \textrm{ on } Z_{d,u}. \]
This is the sum of the indicator variables over all of the closed points of a fixed degree, and thus we can consider $X_k$ as an element of $A(Y)$:
\[ X_k := \sum_{y\in Y  \textrm{ closed of degree k}} X_y. \]

The random variables in $\QQ[X_1, X_2,...]$ are, in a precise sense, the random variables coming from the universal family $Z_{d}$: there is a map
\begin{eqnarray*}
K_0(\Loc_{\QQ_l} U_d)' & \rightarrow &\Map(U_d(\FF_q), \QQ) \\
K & \mapsto &  u \mapsto \tr \Frob_q \actsonr K_{\overline{u}}
\end{eqnarray*} 
and the random variable $X_1$ is the image of $[Rf_* \QQ_l]$ under this map. From the $\lambda$-ring structure on $K_0(\Loc_{\QQ_l} U_d)' $ we obtain
\begin{equation}\label{eqn:lambda-to-finite-set} \Lambda_\QQ \rightarrow  \Map(U_d(\FF_q), \QQ) \end{equation}
sending $h_k$ to
\[ u \mapsto \# \Sym^k Z_{d,u}(\FF_q). \]
It is a standard computation that $p_k'$ maps to the random variable attached to $X_k$. Because the $p_k'$ form a polynomial basis for $\Lambda_\QQ$, this lifts to a map 
\begin{eqnarray*}
\Lambda_\QQ & \rightarrow & A(Y)  \\
p_k' & \mapsto & X_k. 
\end{eqnarray*}
which induces an isomorphism between $\Lambda_\QQ$ and $\QQ[X_1, X_2,...]$. 

From Theorem \ref{theorem:ClosedPointsInd}, we deduce that the the $X_k = (p_k', [Z_{d}/U_d])$ are asymptotically independent with asymptotic binomial distributions, the sum over the closed points of degree $k$ of Bernoulli random variables that are 1 with probability 
\[ \frac{q^{(n) \cdot k} -1}{q^{(n+1) \cdot k}-1}. \]
By abuse of notation, if we write also $\mu_d$ for the pullback of $\mu_d$ to $\Lambda_\QQ$ via (\ref{eqn:lambda-to-finite-set}), then the asymptotic binomial distribution can be written as
\[ \lim_{d \rightarrow \infty} \EE_{\mu_d}\left[ (1+t)^{(p_k', [Z_d/U_d])} \right] = \left(1+ \frac{q^{nk} - 1}{q^{(n+1)k}-1} t \right)^{\# \textrm{ closed points of degree $k$ on } Y}, \]
and the asymptotic independence implies
\[ \lim_{d\rightarrow \infty} \EE_{\mu_d}\left [ \prod_k (1+t_k)^{(p_k', [Z_d/U_d]} \right] = \prod_k \lim_{d\rightarrow \infty} \EE_{\mu_d}\left [ (1+t_k)^{(p_k', [Z_d/U_d])} \right]. \]
Thus, we have proven Theorem \ref{thm:StableGeomPointCount}.

\subsection{Cohomological stabilization}
\label{subsec:PCCohomStab}
In this subsection we prove Theorem \ref{thm:StabCohomPointCount}.

By Poonen \cite{Poonen-Bertini},
\[ \lim_{d \rightarrow \infty} \#U_d(\FF_q)/q^{\dim U_d} = \zeta_Y(n+1)^{-1}, \]
and thus we may normalize by $\# U_d(\FF_q)$ instead of $q^{\dim U_d}$. Applying the Grothendieck-Lefschetz fixed point theorem to the numerator, we are then studying
\[ \lim_{d\rightarrow \infty} \frac{1}{\# U_d(\FF_q)} \sum_{u \in U_d(\FF_q)} \Tr \Frob_q \actsonr \Vc_{\van,\QQ_l, \overline{u}}^{\pi_{\tau,d}}\]
where $\overline{u}$ is the $\overline{\FF_q}$ point obtained by composition of $u$ with $\FF_q \hookrightarrow \overline{\FF_q}$.

In particular, if we consider the random variable $X_\tau$ in 
\[ \Map(U_d(\FF_q), \QQ) \]
given by
\[ X_\tau(u) = \Tr \Frob_q \actsonr \Vc_{\van,\QQ_l, \overline{u}}^{\pi_{\tau,d}} \]
then we are studying the asymptotic behavior of 
\[ \EE_{\mu_d}(X_\tau). \]
Our goal now is to deduce the stabilization of this quantity from Theorem \ref{thm:StableGeomPointCount}. 

The $\QQ$-probability measure on $\Lambda_\QQ$ given by
\[ g  \mapsto \EE_{\mu_d}[(g, [Z_d/U_d])] \]
is the same as the measure of the previous section.  It naturally extends to a $K_0(\Loc_{\QQ_l} \FF_q)$-probability measure on $\Lambda_{K_0(\Loc_{\QQ_l} \FF_q)}$ with values in $\QQ$ and by Theorem~\ref{thm:StableGeomPointCount} the measures $\mu_d$ on $\Lambda_{K_0(\Loc_{\QQ_l} \FF_q)}$ converge to a measure $\mu_\infty$. 

Our strategy now is clear: we will construct an element $s_{\tau,Y}$ of $\Lambda_{K_0(\Loc_{\QQ_l} \FF_q)}$ such that for all $d$, 
\[ (s_{\tau,Y}, [Z_d / U_d]) = [\VanQl^{\pi_{\tau,d}}], \]
which implies 
\[ \EE_{\mu_d}[s_{\tau, Y}] = \EE[X_\tau]. \] 

We denote
\begin{multline*}
[Y^\old] := \sum_{i < n-1} (-1)^i \cdot \left( [H^i(Y_{\overline{\FF_q}}, \QQ_l)] +  [H^{2n - i}(Y_{\overline{\FF_q}}, \QQ_l)(1)] \right) \\ 
+ (-1)^{n-1} [H^{n-1}(Y_{\overline{\FF_q}}, \QQ_l)] \in K_0(\Loc_{\QQ_l} \FF_q).
\end{multline*}
By Equation \ref{eqn:VanCohomDSum-ladic} and the weak Lefschetz theorem, it is the constant part of the cohomology of $Z_{d}$: in $K_0(\Loc_{\QQ_l} U_d)$, we have the identity
\[ [Z_d/U_d] = [Y^\old] + [\VanQl]. \]

Recall that $[\VanQl^{\pi_{\tau,d}}] = (s_\tau, [\VanQl])$ for an appropriate Schur polynomial $s_\tau \in \Lambda$ as in Proposition \ref{prop:SurjLambdaSchur}. 

We observe that $(p_k', \underline{\;\;})$ is additive, so that 
\[ (p_k', [\VanQl]) = (p_k', [Z_d]) - (p_k', [Y^\old]) \]
and thus
\[ (p_k', [\VanQl]) = (p_k' - (p_k', [Y^\old]), [Z_d]). \]

In particular, if we let $s_{\tau,Y}$ be the element in $\Lambda_{K_0(\Loc_{\QQ_l} \FF_q)}$ given by expressing $s_\tau$ as a polynomial (with $\QQ$-coefficients) in $p_k'$  and then substituting 
\[ p_{k,Y}' := p_k' - (p_k', [Y^\old]) \]
for $p_k'$, we obtain
\[ \EE_{\mu_d} [ s_{\tau,Y} ] = \frac{1}{\#U_d(\FF_q)} \sum_{u \in U_d(\FF_q)} \Tr \Frob \actsonr \Vc^{\pi_{\tau,d}}_{\van, \QQ_l, \overline{u}} \]

This stabilizes to $\EE_{\mu_\infty}[s_{\tau,Y}]$ as $d \rightarrow \infty$. Moreover, since the $(p_k', [Y^\old])$ are constants and the $p_k'$ are independent for $\mu_\infty$, so are the $p_{k,Y}'$. The asymptotic falling moment generating function for $p_{k,Y}'$ is given by multiplying the asymptotic function for $p_{k}'$ in Theorem \ref{thm:StableGeomPointCount} by the moment generating function for the constant random variable $(p_k', [Y^\old])$ which is 
\[ (1+t)^{-\Tr \Frob \actsonr (p_k', [Y^\old])}=(1+t)^{-\sum_{d|k} \mu(k/d) \Tr \Frob^k \actsonr [Y^\old]}. \]
In particular, the moments of $p_{k,Y}'$ are, for fixed dimension $n$, given by universal formulas that are rational functions of $q$ and symmetric functions  of the eigenvalues of Frobenius acting on the cohomology of $Y$. 

Thus we obtain Theorem \ref{thm:StabCohomPointCount} (where $\tau$ here plays the role of $\sigma$ in the statement of Theorem \ref{thm:StabCohomPointCount}):
 \[ \lim_{d \rightarrow \infty} q^{-\dim U_d}\sum_i (-1)^i \Tr \Frob_q \actsonr H_c^\bullet (U_{d,\overline{\FF_q}}, \VanQl^{\pi_{\tau, d}}) \]
stabilizes to $\zeta_Y(n+1)^{-1}\EE_{\mu_\infty}[s_{\tau,Y}]$, which, for fixed $\tau$ and $n$, is given by an explicit universal formula that is a rational function in $q$ and symmetric functions of the eigenvalues of Frobenius acting on the cohomology of $Y$ (this comes from expressing $s_{\tau,Y}$ as a polynomial in $p_{k,Y}'$ and then using independence to compute its expectation in terms of the moments of $p_{k,Y}'$, which are functions of this form -- cf. Appendix \ref{appendix:AlgAndComp} for more details and an example). 

\section{Hodge theoretic results}\label{sec:BettiResults}
In this section we prove Theorem \ref{thm:StabCohomHodge} and Theorem \ref{thm:StabGeomMot}. We start by generalizing the proof of \cite[Proposition 3.5]{VakilWood-Discriminants} to obtain a weak version of Theorem \ref{thm:StabGeomMot}, Proposition \ref{prop:weakStab} below, which implies the limits in question exist but gives a more complicated formula for their values. To obtain the explicit formulas of Theorem \ref{thm:StabGeomMot}, we construct a generating function for these coefficients using the power structure on the Grothendieck ring of varieties, mirroring the strategy used for configuration spaces in \cite{Howe-MRV1}. The situation is more complicated in this case because in order to make the motivic probabilities $\frac{\bbL^n - 1}{\bbL^{n+1}-1}$ appear after evaluation at a power of $\bbL$, we are forced to consider a power series with non-effective coefficients. However, the geometric description of the power structure given in \cite{GLM-PowerStructure} only applies when the series has effective coefficients. Recent work of Bilu \cite{bilu:thesis, bilu:thesis-article} on motivic Euler products gives a systematic way to compute powers of non-effective series, and combining this with a combinatorial lemma we are able to obtain the desired identity. 

We use the setup of Section \ref{sec:vancohom} with $\KK=\CC$. 

\subsection{Partitions and a combinatorial lemma}\label{subsec:partitions-and-combinatorial-lemma}
Following Vakil-Wood \cite{VakilWood-Discriminants}, we let $\calQ$ denote the set of ordered partitions, i.e. of tuples of integers $(m_1, m_2, \ldots )$ such that for some $k \in \bbN$, $m_i = 0$ for all $i > k$ and $m_i > 0$ for all $i < k$. For $m_1, \ldots, m_k >0$, we will denote by $(m_1, m_2, \ldots, m_k)$ or $1^{m_1} 2^{m_2} \ldots k^{m_k}$ the tuple $(m_1, m_2, \ldots, m_k, 0, 0, \ldots)$. The tuple $(0,0,0, \ldots)$ is in $\calQ$, and we will sometimes write it as $\emptyset$.

Let $\calQ_0 := \dsum_{i \in \bbN} \bbZ_{\geq 0}$, the set of ordered partitions allowing zero, i.e. of tuples of integers $(m_1, m_2, \ldots )$ such that $m_i$ is zero for $i$ sufficiently large. We use similar notation for elements of $\calQ_0$. 

For $\mu = (m_1, m_2, \ldots) \in \calQ_0$, we define $|\mu|=\sum m_i$ and $|| \mu || = \# \{i \; | \; m_i > 0 \}$. We will sometimes write $\mu(i)$ for the $i$th component $m_i$, and $\underline{t}^\mu$ for the monomial $\prod_{i} t_i^{\mu(i)}$ in indeterminates $t_1, t_2, \ldots$. 

\begin{example}
For 
\[ \mu = 1^2 2^1 3^5 = (2, 1, 5) = (2,1,5,0,0\ldots) \in \calQ, \textrm{ and } \]
\[ \tau = 2^1 = (0,1)=(0,1,0,0\ldots) \in \calQ_0 - \calQ, \]
we have
\[ |\mu|=8,\, |\tau|=1,\, ||\mu||=3,\, ||\tau||=1,\, \mu(3)=5,\, \tau(1)=0,\, \underline{t}^\mu=t_1^2 t_2 t_3^5,\, \textrm{ and } \underline{t}^\tau=t_2. \]
\end{example}

There is a natural retraction map $c: \calQ_0 \rightarrow \calQ$ which restricts to the identity on $\calQ$ given by removing intermediate zeroes.

\begin{example}
\[ c(1^2 3^5) = c ( (2,0,5) )= 1^2 2^5 = (2,5). \]
\end{example}

\begin{remark}
A helpful visualization is to consider an element $\mu \in \calQ_0$ as a pile of blocks with columns indexed by $\NN$, and with the $i$th column containing $m_i$ blocks. The elements of $\calQ$ are piles with no gaps between non-empty columns, $|\mu|$ is the total number of blocks in the pile, and $||\mu||$ is the total number of non-empty columns. The contraction $c(\mu)$ is given by sliding all of the non-empty columns as far to the left as possible to close the gaps.  
\end{remark}

We will consider the \emph{decomposition sets}
\[ \calQ_0^{n,*} := \{ (\mu_1, \ldots , \mu_n) \in \calQ_0^n \; | \; \mu_1 + \ldots + \mu_n \in \calQ \}. \]
There is a summation map $\calQ_0^{n,*} \rightarrow \calQ$, $(\mu_1, \ldots, \mu_n) \rightarrow \mu_1 + \ldots + \mu_n$, and a contraction map $c_n: \calQ_0^{n,*} \rightarrow \calQ^{n}$, $(\mu_1, \ldots, \mu_n) \rightarrow (c(\mu_1), \ldots, c(\mu_n)).$  

\begin{remark}\label{remark:piles}
In this remark we give a visual description of the pre-image $c_n^{-1}(\mu_1', \mu_2', \ldots \mu_n')$ and it's behavior under the summation map; it may be helpful to keep this interpretation in mind for Lemma \ref{lemma:key-combinatorics-of-piles} below. 

Given $(\mu_1', \mu_2', \ldots \mu_n') \in \calQ^n$, we can build new partitions in $\calQ$ by starting with the pile $\mu_1'$, then, for each column in $\mu_2'$, either slipping it between two columns or putting it on top of an already existing column, always moving from left to right. After inserting all of the columns from $\mu_2'$, we obtain a new element of $\calQ$, and we may repeat the process to add the columns from $\mu_3'$, etc.; we will end up with a pile $\mu \in \calQ$. The process used to build $\mu$ from $\mu_1',\ldots, \mu_n'$ will determine (and is determined by) a unique decomposition $\mu=\mu_1 + \ldots \mu_n$ with $\mu_i \in \calQ_{0}$ and $c(\mu_i)=\mu_i'$, and thus describes an element of $c_n^{-1}(\mu_1', \mu_2', \ldots \mu_n')$ mapping to $\mu$ under summation.  
\end{remark}

We will use the following lemma in our proof of Theorem \ref{thm:StabGeomMot} below. 
\begin{lemma}\label{lemma:key-combinatorics-of-piles}
Given $(\mu_1', \ldots, \mu_n') \in \calQ^n$, 
\[ \sum_{(\mu_1, \ldots, \mu_n) \in c_n^{-1} (\mu_1', \ldots, \mu_n')}  (-1)^{|| \mu_1 + \ldots + \mu_n ||} = (-1)^{||\mu_1'|| + \ldots + ||\mu_n'||} \]
\end{lemma}

\begin{proof}
The case $n=1$ is evident as $c_1^{-1}((\mu_1'))=(\mu_1')$. We will show the case $n=2$ directly, and then proceed by induction for larger $n$.

We now consider the case $n=2$. To give an element $(\mu_1, \mu_2)$ in $c_2^{-1}(\mu_1', \mu_2')$ is the same as to give:
\begin{enumerate}
\item $j \leq \min(||\mu_1'||, ||\mu_2'||)$
\item A subdivision of $\{1, \ldots, ||\mu_1'|| + ||\mu_2'|| - j\}$ into a subset of size $j$, a subset of size $||\mu_1'|| - j$, and a subset of size $||\mu_2'|| - j$.
\end{enumerate}
Here $j$ corresponds to the number of $k$ such that $\mu_{1}(k)$ and $\mu_2(k)$ are both non-zero; given this information, to determine $\mu_1$ and $\mu_2$ from $\mu_1'$ and $\mu_2'$, it suffices to pick $j$ spots for both to be non-zero, $||\mu'_1|| - j$ for only $\mu_1$ to be non-zero, and $||\mu'_2|| - j$ for only $\mu_2$ to be non-zero. Furthermore, for the resulting $(\mu_1, \mu_2)$, we have 
\[ (-1)^{||\mu_1 + \mu_2||} = (-1)^{||\mu_1'|| + ||\mu_2'||-j} = (-1)^{||\mu_1'|| + ||\mu_2'||} (-1)^j. \]
Writing $||\mu_1||=a$, $||\mu_2||=b$, and assuming $a\leq b$, the desired identity reduces to
\[ \sum_{j=0}^a \binom{b}{j}\binom{a+b-j}{a-j}(-1)^j = 1. \]
This identity is established by plugging in $X=b$, $Y=-b-1$ into the Chu-Vandermonde identity\footnote{We thank an anonymous referee for suggesting this application of the Chu-Vandermonde identity, which gives a shorter proof of the $n=2$ case than the direct combinatorial argument based on Remark \ref{remark:piles} that appeared in an earlier draft.} 
\[ \sum_{j=0}^a \binom{X}{j} \binom{Y}{a-j} = \binom{X+Y}{a},  \]
and simplifying with the identities
\begin{align*} \binom{-b-1}{a-j}&=\frac{(-b-1)(-b-2)\ldots(-b-a+j)}{(a-j)!}=(-1)^{a-j}\binom{a+b-j}{a-j}, \\ 
\textrm{ and } \binom{-1}{a}&=\frac{(-1)(-2)\ldots(-a)}{a!}=(-1)^a. \end{align*}

We now assume $n \geq 3$, and that the identity is known for $m < n$. We can rewrite the sum as
\[ \sum_{(\mu_1, \ldots, \mu_{n-1}) \in c_{n-1}^{-1} (\mu_1', \ldots, \mu_{n-1}')} \sum_{ (\nu, \mu_n) \in c_2^{-1} (\mu_1 + \ldots + \mu_{n-1}, \mu_n') } (-1)^{||\nu + \mu_n||}. \]
Applying the induction hypothesis to the interior sum we find that this is equal to
\begin{multline} \sum_{(\mu_1, \ldots, \mu_{n-1}) \in c_{n-1}^{-1} (\mu_1', \ldots, \mu_{n-1}')} (-1)^{||\mu_1 + \ldots \mu_{n-1}|| + ||\mu_n'||} = \\
(-1)^{||\mu_n'||} \sum_{(\mu_1, \ldots, \mu_{n-1}) \in c_{n-1}^{-1} (\mu_1', \ldots, \mu_{n-1}')} (-1)^{||\mu_1 + \ldots \mu_{n-1}||}, 
\end{multline}
and we conclude by applying the induction hypothesis to the remaining sum. 
\end{proof}

\subsection{A first stabilization formula for configuration spaces}
\newcommand{\bc}{\mathbf{c}}

The proof of the following proposition follows the same strategy as \cite[Proposition 3.5]{VakilWood-Discriminants}, which corresponds to $\tau=\emptyset$.  

\begin{proposition}\label{prop:weakStab}
For $\tau \in \calQ_{0},$ the following identity holds in $\motcomp$: 
\begin{multline}\label{eqn:ConfLimitFormula}
 \lim_{d\rightarrow \infty} \frac{[\Conf^\tau_{U_d} Z_d]}{\LL^{\dim U_d}} = \\
 \sum_{\mu \in \calQ} (-1)^{||\mu||} \sum_{\substack{\tau_1 + \tau_2 = \tau,\;\;  \sum \mu_i + \mu' = \mu, \\ |\mu_i|=\tau_2(i),\;\; \tau_1, \tau_2, \mu_i, \mu' \in \calQ_{0}}}\LL^{-|\tau_1| -|\mu|(n+1)} [\Conf^{\tau_1 \cdot \mu_1 \cdot \mu_2 \cdot \ldots \cdot \mu'} Y]. 
 \end{multline}
\end{proposition}
\begin{proof}

Let
\[ V_d:= \AA(\Gamma(Y, \Oc(d))). \]
Given partitions $\tau$ and $\mu$ we denote by $W^\tau_{\geq \mu}$ the closed subvariety of 
\[ V_d \times \Conf^\tau Y \times \Conf^\mu Y \]
consisting of $(s, \bc_{\tau}, \bc_{\mu} )$ such that $s$ vanishes at the points in $\bc_{\tau}$ and such that $s$ is singular at the points in $\bc_{\mu}$. We denote by $W^\tau_{\mu}$ the constructible subset such that $s$ vanishes at the points in $\bc_{\tau}$ and is singular at \emph{exactly} the points in $\bc_{\mu}$. In particular, we have
\[ W^\tau_{\emptyset} = \Conf^\tau_{U_d} Z_d. \]

We claim that, for any $k \in \ZZ_{\geq 1}$,
\begin{equation}\label{eqn:WgeqExpresion}
[W^\tau_{\geq \mu}] = [W^\tau_{\mu}] + [W^\tau_{\mu\cdot *}] + [W^\tau_{\mu \cdot *^2}] + ... + [W^\tau_{\mu \cdot *^{k-1}}] + [(W^\tau_{\geq \mu})_k ]. 
\end{equation}
where $(W^\tau_{\geq \mu})_k$ denotes the image of $W^\tau_{\geq \mu \cdot *^k}$ in $W^\tau_{\geq \mu}$ under the natural map 
\[ \pi_k: W^\tau_{\geq \mu \cdot *^k} \rightarrow W^\tau_{\geq \mu}. \]
Equation \ref{eqn:WgeqExpresion} holds because for any $k$, we have an equality in the $K_0(\Var/\CC)$
\[ [W^\tau_{\mu \cdot *^k}] = [\pi_k(W^\tau_{\mu \cdot *^k})]. \]
Here $\pi_k(W^\tau_{\mu \cdot *^k})$ is the constructible subset of $W^\tau_{\geq \mu}$ where the section is singular at exactly $k$ points outside of those marked by $\mu$. 
 
We may decompose $\Conf^\tau Y \times \Conf^\mu Y$, and thus, $W^\tau_{\geq \mu}$, into locally closed sets determined by the overlap of the configurations $\bc_{\tau}$ and $\bc_{\mu}$. We enumerate the possibilities for the overlap: first we must choose an ordered decomposition $\tau=\tau_1 + \tau_2$; the partition $\tau_1$ marks the points which are in $\bc_{\tau} - \bc_{\mu}$, and the partition $\tau_2$ marks the points which are in $\bc_{\tau} \cap \bc_\mu$.  Then, writing $\tau_2=(k_1, \ldots, k_l)$, we must choose an ordered decomposition 
\[ \mu = \mu_1 + \ldots \mu_l + \mu' \]
where $|\mu_i|=k_i$; the partition $\mu_i$ corresponds to the points in $\bc_{\tau} \cap \bc_\mu$ and where the marking from $\tau$ is with label $a_i$, and the partition $\mu'$ corresponds to the points in $\bc_{\mu} - \bc_\tau$. We denote the corresponding subvariety of $\Conf^\tau Y \times \Conf^\mu Y$ by $w_{\tau_\bullet, \mu_\bullet}$, and its preimage in $W^\tau_{\geq_\mu}$ by $W^{\tau_\bullet}_{\geq\mu, \mu_\bullet}$. The natural map
\begin{equation}\label{eqn:conf-iso-overlap} \Conf^{\tau_1 \cdot \mu_1 \cdot \mu_2 \cdot \ldots \cdot \mu'} Y \rightarrow \Conf^{\tau} Y \times \Conf^\mu Y \end{equation}
induces an isomorphism between $\Conf^{\tau_1 \cdot \mu_1 \cdot \mu_2 \cdot \ldots \cdot \mu'} Y$ and its scheme theoretic image $w_{\tau_\bullet, \mu_\bullet}$. 

\begin{example}
If $\tau=a^1 b^1$, and $\mu=1^3$, then, we might choose 
\[ \tau_1 = a^1 , \tau_2=b^1, \mu_1=\emptyset, \mu_2= 1^1, \mu'=1^2. \]
For this choice, $W^{\tau_\bullet}_{\geq \mu, \mu_\bullet}$ consists of hypersurface sections with a non-singular point marked by the label $a$, a singular point marked by the labels $a$ (from $\tau$) and $1$ (from $\mu$), and a pair of singular points marked only by the label $1$ from $\mu$.
\end{example}

We denote by $\Fil^\bullet$ the decreasing dimension filtration on $K_0(\Var)[\LL^{-1}]$, so that $\Fil^k$ is spanned by classes of ``dimension $-k$", i.e. of the form $[X]/\LL^m$ where $\dim X \leq m-k$. We also denote by $\Fil^\bullet$ the induced filtration on $\motcomp$.  

Using the decomposition of $W_{\geq \mu}^\tau$ above, we obtain:

\begin{lemma} \label{lem:VecBundleParts}
Fix a partition $\tau$, and a positive integer $m$. Then, there exists an $N>0$ such that for all $d>N$, and $|\mu|\leq m$,
\begin{equation} \label{eqn:VecBundleParts} \frac{[W^\tau_{\geq \mu}]}{\bbL^{\dim U_d}} = \sum_{\substack{\tau_1 + \tau_2 = \tau,\;\;  \sum \mu_i + \mu' = \mu, \\ |\mu_i|=\tau_2(i),\;\; \tau_1, \tau_2, \mu_i, \mu' \in \calQ_{0}}}\LL^{-|\tau_1| -|\mu|(n+1)} [\Conf^{\tau_1 \cdot \mu_1 \cdot \mu_2 \cdot \ldots \cdot \mu'} Y] 
\end{equation}
in $K_0(\Var)[\bbL^{-1}]$, and for all $\mu, k$ such that $|\mu| + k = m$, 
\[ [(W^\tau_{\geq \mu})_k] / \LL^{\dim U_d} \equiv 0 \mod \Fil^{ -(n-1) \cdot |\tau| + m} \] 
\end{lemma}
\begin{proof}[Proof of lemma]
As above, we can write
\[ W^{\tau}_{\geq \mu}  = \bigsqcup_{\substack{\tau_1 + \tau_2 = \tau,\;\;  \sum \mu_i + \mu' = \mu, \\ |\mu_i|=\tau_2(i),\;\; \tau_1, \tau_2, \mu_i, \mu' \in \calQ_{ 0}}} W^{\tau_\bullet}_{\geq\mu, \mu_\bullet}. \]

Arguing as in \cite[Lemma 3.2]{VakilWood-Discriminants}, we find that for $d$ sufficiently large, $W^{\tau_\bullet}_{\geq \mu, \mu_\bullet}$ is a vector bundle of rank $\dim V_d - |\tau_1| - |\mu|(n+1)$ over $w_{\tau_\bullet, \mu_\bullet}$: here the quantity $|\tau_1| + |\mu|(n+1)$ being subtracted comes from the fact that marking a point as being on a hypersurface section imposes one linear condition on $V_d$ and marking a point as singular imposes $n+1$ linear conditions on $V_d$.  

Thus, using the isomorphism induced by (\ref{eqn:conf-iso-overlap}) between $\Conf^{\tau_1 \cdot \mu_1 \cdot \mu_2 \cdot \ldots \cdot \mu'} Y$ and  $w_{\tau_\bullet, \mu_\bullet}$,  we find
\[ [W^{\tau_\bullet}_{\geq \mu, \mu_\bullet}] = \LL^{\dim{V_d} - |\tau_1| - |\mu|(n+1)} \cdot [\Conf^{\tau_1 \cdot \mu_1 \cdot \mu_2 \cdot \ldots \cdot \mu'} Y] . \]
Since $\dim U_d = \dim V_d$, we obtain the equality~(\ref{eqn:VecBundleParts}) by summing these expressions for $[W^{\tau_\bullet}_{\geq \mu, \mu_\bullet}]$ and dividing by $\LL^{\dim U_d}$. 

To prove the final statement of the lemma, we note that for any choice of $\tau_\bullet$ and $\mu_\bullet$ as above 
\begin{multline*} - |\tau_1| - |\mu|(n+1) + \dim \Conf^{\tau_1 \cdot \mu_1 \cdot \mu_2 \cdot \ldots \cdot \mu'} Y = \\  -|\tau_1| - |\mu|(n+1) + (|\tau_1| + |\mu|)(n) = |\tau_1|(n-1) - |\mu| \end{multline*}
and thus $\dim W_{\geq \mu}^\tau - \dim U_d \leq |\tau|(n-1) - |\mu|$. We conclude the final statement by replacing $\mu$ with  $\mu \cdot *^k$ and using the corresponding dimension bound on $\pi_k( W_{\geq \mu \cdot *^k}^\tau)$.  
\end{proof}

We can now conclude the proof of Proposition \ref{prop:weakStab}. Fix a $m$, and let $N$ be large enough such that Lemma \ref{lem:VecBundleParts} holds for our fixed $\tau$ and $m$. Let $d \geq N$.

Using Equation \ref{eqn:WgeqExpresion} iteratively to replace terms of the form $[W^\tau_\mu]$ with terms of the form $[W^\tau_{\geq \mu}]$, we find that, 
\begin{multline}\label{eqn:expansion-with-error-term}
\frac{[\Conf^\tau_{U_d} Z_d / U_d ]}{\LL^{\dim U_d}}  = \LL^{- \dim U_d} \left( [W_{\geq \emptyset}^\tau] - [W_{ *}^\tau] - ... - [W_{ *^{m-1}}^\tau] - [(W_{\geq \emptyset}^\tau)_{m}] \right) \\
= \LL^{- \dim U_d}  ( ([W_{\geq \emptyset}^\tau]) - ([W_{\geq *}^\tau] - [W_{*\cdot\star}^\tau] - ... - [W_{*\cdot\star^{m-2}}^\tau] - [(W_{\geq*}^\tau)_{m-1}]) \\ - ([W_{\geq *^2}^\tau] - [W_{*^2\cdot\star}^\tau] - ... - [W_{*^2\cdot\star^{m-3}}^\tau] - [(W_{\geq*^2}^\tau)_{m-2}])   - ... - [(W^\tau_{ \geq \emptyset})_{m}]) \\
= ... = \LL^{-\dim U_d} \left( \sum_{ \mu \in \calQ, |\mu|< m} (-1)^{||\mu||} \cdot \left( [W^\tau_{\geq \mu}] - [(W^\tau_{\geq {\mu}})_{m - |\mu|}]\right) \right).
\end{multline}
The second part of Lemma \ref{lem:VecBundleParts} implies that 
\[ \frac{[(W^\tau_{\geq \mu})_{m-|\mu|}]}{\LL^{\dim U_d}} \equiv 0 \mod  \Fil^{-(n-1) \cdot |\tau| + m}. \]
Thus, (\ref{eqn:expansion-with-error-term}) gives
\[ \frac{[\Conf^\tau_{U_d} Z_d / U_d ]}{\LL^{\dim U_d}}  = \LL^{-\dim U_d} \sum_{ \mu \in \calQ, |\mu| < m} (-1)^{||\mu||} \cdot [W^\tau_{\geq \mu}] \textrm{ mod } \Fil^{-(n-1) \cdot |\tau| + m}\]

We obtain the formula (\ref{eqn:ConfLimitFormula}) for the limit by plugging in the first part of Lemma \ref{lem:VecBundleParts} for each term of the sum and then taking $m \rightarrow \infty$ (and thus, also $d\rightarrow \infty$, as  the identity holds only for $d$ sufficiently large, depending on $m$). 

\end{proof}

\subsection{A generating function and the proof of Theorem \ref{thm:StabGeomMot}}
Recall from \cite{GLM-PowerStructure} (cf. also \cite{Howe-MRV1}) that there is a power structure on the Grothendieck ring of varieties, which gives a way to make sense of expressions of the form $f^a$
for $f$ a generalized power series with constant term $1$ and $a \in K_0(\Var/\CC)$ (or more generally $K_0(\Var / X)$ for a variety $X/\CC$) that satisfies many of the formal properties one would expect from the notation. The power structure gives a useful way to organize computations in the Grothendieck ring with generalized configuration spaces, as demonstrated, e.g., in the prequel \cite{Howe-MRV1}, and we will use it again here in a similar way. To avoid confusion with the naive exponential, we use the notation of \cite{Howe-MRV1} and write $f^{_\Pow a}$ to denote exponentials in the power structure. In the next section, we will show:
\begin{proposition}\label{prop:generating-identity} For any variety $X / \mathbb{C} $
\begin{multline}
\label{eqn:pro-gen-func}
(1 - s\cdot(1+ t_1 + t_2 +...) + z \cdot (t_1 + t_2 + ...))^{_\Pow [X]} = \\
\sum_{\tau \in \calQ_{0} } \sum_{\mu \in Q} (-1)^{||\mu||} \sum_{\substack{\tau_1 + \tau_2 = \tau,\;\;  \sum \mu_i + \mu' = \mu, \\ |\mu_i|=\tau_2(i),\;\; \tau_1, \tau_2, \mu_i, \mu' \in \calQ_{ 0}}} [\Conf^{\tau_1 \cdot \mu_1 \cdot \ldots \cdot \mu_m \cdot \mu'} X] s^{|\mu|} z^{|\tau_1|} \underline{t}^\tau 
\end{multline}
\end{proposition}

Assuming this proposition, we prove Theorem \ref{thm:StabGeomMot}: 
\begin{proof}[Proof of Theorem \ref{thm:StabGeomMot}]

The elements $c_\tau \in \Lambda$ described in Example \ref{example:RelConfStab} form a basis for $\Lambda$ as a $\ZZ$-module (cf. \cite[Subsection \ref{MRV1-subsec:GrothPow}]{Howe-MRV1}). Thus, to show that for any $f \in \Lambda$,
\[ \lim_{d\rightarrow \infty} \EE_{\mu_d}[(f,[Z_d/U_d])] \] 
exists in $\motcomp$, it suffices to verify that for any $\tau$,
\begin{equation*}
\lim_{d\rightarrow \infty} \EE_{\mu_d}[(c_\tau, [Z_d / U_d])] 
\end{equation*}
exists in $\motcomp$. Because 
\[ (c_\tau, [Z_d/U_d]) = [\Conf^\tau_{U_d} Z_d / U_d], \]
we must verify
\begin{equation}\label{eqn:ConfSpaceLimit}
\lim_{d\rightarrow \infty} \frac{[\Conf^\tau_{U_d} Z_d]}{[U_d]} 
\end{equation}
exists in $\motcomp$. We recall (cf. Remark \ref{rmk:VWMotZeta}) that, by \cite{VakilWood-Discriminants}, 
\begin{equation}\label{eqn:VWLimDisc}
\lim_{d\rightarrow \infty} \frac{[U_d]}{\LL^{\dim U_d}} = Z_Y\left(\LL^{-(n+1)}\right)^{-1} \in \motcomp. 
\end{equation}
Thus, combining with Proposition \ref{prop:weakStab} we see the limit exists and moreover we obtain the formula
\begin{multline} \label{eqn:ExpectationLimitFormula} \lim_{d\rightarrow \infty} \EE_{\mu_d}[(c_\tau, [Z_d / U_d])] =  Z_Y\left(\LL^{-(n+1)}\right) \cdot \\ \sum_{\mu \in \calQ} (-1)^{||\mu||} \sum_{\substack{\tau_1 + \tau_2 = \tau,\;\;  \sum \mu_i + \mu' = \mu, \\ |\mu_i|=\tau_2(i),\;\; \tau_1, \tau_2, \mu_i, \mu' \in \calQ_{ 0}}}\LL^{-|\tau_1| -|\mu|(n+1)} [\Conf^{\tau_1 \cdot \mu_1 \cdot \mu_2 \cdot \ldots \cdot \mu'} Y]. \end{multline}

Now, since the $c_\tau$ form a basis for $\Lambda_\QQ$, to verify the asymptotic independence and asymptotic distribution of the random variables $(p_k', [Z_d/U_d])$, it suffices to show that for each $\tau$, 
\[ \lim_{d\rightarrow \infty} \EE_{\mu_d}[(c_\tau, [Z_d / U_d])] \]
is equal to the value predicted by the asymptotic independence and distribution (cf. \cite[Section 5.3]{Howe-MRV1}). Concretely, we have 
\begin{align*} \sum_{\tau \in \calQ_{0}} \underline{t}^\tau \lim_{d\rightarrow \infty} \EE_{\mu_d}[(c_\tau, [Z_d/U_d])] & = \lim_{d \rightarrow \infty} \EE_{\mu_d}\left[ (1 + t_1 + t_2 + \ldots)^{_\Pow [Z_d / U_d]} \right] \\ & = \lim_{d\rightarrow \infty}\EE_{\mu_d}\left[ \prod_k (1 + t_1^k + t_2^k + \ldots)^{(p_k', [Z_d / U_d])}\right], 
\end{align*}
where the second equality follows from \cite[Lemma \ref{MRV1-lem:EulerCoeffs}]{Howe-MRV1}, 
and thus it suffices to show that 
\[  \sum_{\tau \in \calQ_{0}} \underline{t}^\tau \lim_{d\rightarrow \infty} \EE_{\mu_d}[(c_\tau, [Z_d/U_d])] = \prod_k \left(1 + \frac{\bbL^{nk}-1}{\bbL^{(n+1)k}-1}(t_1^k + t_2^k + ...)\right)^{(p_k', [Y])}. \]

We will arrive at this identity by manipulating the equation of Proposition \ref{prop:generating-identity}. First, we may divide both sides of the equation in Proposition \ref{prop:generating-identity} by $(1-s)^{_\Pow [Y]}$ to obtain
\begin{multline}
\label{eqn:divided-gen-function-identity}
\left(1 + \frac{z-s}{1-s}(t_1 + t_2 + ...)\right)^{_\Pow [Y]} = (1-s)^{_\Pow -[Y]} \cdot \\  \\
\sum_{\tau \in \calQ_{0} } \sum_{\mu \in Q} (-1)^{||\mu||} \sum_{\substack{\tau_1 + \tau_2 = \tau,\;\;  \sum \mu_i + \mu' = \mu, \\ |\mu_i|=\tau_2(i),\;\; \tau_1, \tau_2, \mu_i, \mu' \in \calQ_{ 0}}} [\Conf^{\tau_1 \cdot \mu_1 \cdot \ldots \cdot \mu_m \cdot \mu'} Y] s^{|\mu|} z^{|\tau_1|} \underline{t}^\tau 
\end{multline}

We have $(1-s)^{_\Pow -[Y]}=Z_Y(s)$. Thus, plugging in 
\[ s=\LL^{-(n+1)},\; z= \LL^{-1} \] and comparing with Equations (\ref{eqn:ConfLimitFormula}) and (\ref{eqn:VWLimDisc}), we find that after evaluation the coefficient of $\underline{t}^\tau$ on the right-hand side of (\ref{eqn:divided-gen-function-identity}) is equal to
\[ \lim_{d\rightarrow \infty} \EE_{\mu_d}[(c_\tau, [Z_d/U_d])]. \]
On the other hand, using \cite[Lemma \ref{MRV1-lem:EulerCoeffs}]{Howe-MRV1}, the left-hand side of (\ref{eqn:divided-gen-function-identity}) has a naive Euler product expansion as 
\[ \prod_k \left(1 + \frac{z^k-s^k}{1-s^k}(t_1 + t_2 + ...)\right)^{(p_k', [Y])}. \]
Thus, plugging in 
\[ s=\LL^{-(n+1)},\; z= \LL^{-1} \] 
and simplifying we conclude 
\[ \sum_{\tau \in \calQ_{0}} \underline{t}^\tau \lim_{d\rightarrow \infty} \EE_{\mu_d}[(c_\tau, [Z_d/U_d])] = \prod_k \left(1 + \frac{\bbL^{nk}-1}{\bbL^{(n+1)k}-1}(t_1^k + t_2^k + ...)\right)^{(p_k', [Y])} \]
in $\motcomp \otimes \QQ$, as desired. 

\end{proof}

\subsection{Powers of non-effective series}
\newcommand{\bbC}{\mathbb{C}}
In this section we explain how the results of Bilu \cite{bilu:thesis, bilu:thesis-article} on motivic Euler products can be used to compute powers of non-effective series, and in particular obtain Proposition \ref{prop:generating-identity}. 

Given a variety $X/\bbC$ and elements $a_\tau \in K_0(\Var/ X)$ for $\tau \in \calQ_{0}\backslash\{0\}$, the theory of motivic Euler products defines a power series with constant term 1
\[ \prod_{x \in X} \left(1 + \sum_{\tau \in \calQ_0} a_\tau t^\tau \right) \in K_0(\Var/\bbC)[[t_1, t_2, \ldots]]. \]
As with power structures, the infinite product notation gives an indication of which manipulations are valid with these products -- cf. \cite[3.8-3.10]{bilu:thesis-article} for a precise description of the properties satisfied. Most importantly for us, motivic Euler products are compatible with the power structure: if the classes $a_\tau \in K_0(\Var/X)$ are given by pulling back classes $b_\tau$ in $K_0(\Var/\bbC)$ (i.e., the $a_\tau$ are constant), then 
\begin{equation}\label{eqn:mot-euler-mot-power} \prod_{x \in X} \left(1 + \sum_{\tau \in \calQ_0 \backslash \{0\}} a_\tau \underline{t}^\tau \right) = \left( 1 + \sum_{\tau \in \calQ_0 \backslash \{0\}} b_\tau \underline{t}^\tau \right)^{_\Pow [X]}. \end{equation}
Indeed, when the $b_\tau$ are effective, both sides are immediately seen to agree by comparing the geometric description given in \cite{GLM-PowerStructure} of the coefficients of the power on the right-hand side in terms of generalized configuration spaces with the definition of the motivic Euler product \cite[3.8.1]{bilu:thesis-article} (which is also in terms of generalized configuration spaces when the coefficients are effective). As an arbitrary series can be multiplied by an effective series to obtain another effective series, we conclude using multiplicativity of motivic Euler products\footnote{In \cite{bilu:thesis}, multiplicativity is shown only for effective power series, but the general case appears in the updated version \cite{bilu:thesis-article}.} and multiplicativity of the power structure that both sides always agree.

We will make use of (\ref{eqn:mot-euler-mot-power}) and the definition of motivic Euler products to give a general expression for motivic powers in Theorem \ref{thm:non-effective-power} below, which we will then use to establish Proposition \ref{prop:generating-identity}. Before stating Theorem \ref{thm:non-effective-power}, we recall some notation from \cite{bilu:thesis,bilu:thesis-article} used in the definition of motivic Euler products. 

Given $X/\bbC$, we consider the complete graded algebra 
\[ K_0(\Var / \Sym^\bullet X) := \prod_{k \geq 0} K_0(\Var / \Sym^k X) \]
with termwise addition, and multiplication induced on graded components by composition of the exterior multiplications
\[ \times: K_0(\Var / \Sym^{k_1} X) \times K_0(\Var/ \Sym^{k_2}) \rightarrow K_0(\Var / \Sym^{k_1} X \times \Sym^{k_2} X) \]
with the forgetful maps
\[ K_0(\Var / \Sym^{k_1} X \times \Sym^{k_2} X) \rightarrow K_0(\Var / \Sym^{k_1 + k_2} X ). \]
Denote by $K_0(\Var / \Sym^\bullet X)^1 \subset K_0(\Var / \Sym^\bullet X)$ the subgroup of the multiplicative units whose constant ($k=0$) term is $1 \in K_0(\Var/ \bbC)$. 

\begin{lemma} [Lemma 3.5.1.2 of \cite{bilu:thesis-article}] There is a unique group homomorphism 
\[ S: K_0(\Var / X) \rightarrow K_0(\Var / \Sym^\bullet X)^1 \]
such that for a variety $Y/X$, 
\[ S([Y/X])=(1, [Y/X], [\Sym^2 Y / \Sym^2 X], \ldots ). \]
\end{lemma}

We introduce some related notation: given $a \in K_0(\Var/X)$, we write $S^n a \in K_0(\Var/\Sym^n X)$ for the $n$th component of $S(a)$. Given an indexing set $I$ and, for each $i\in I$, $n_i \in \ZZ_{\geq 0}$ and $a_i \in K_0(\Var /X)$ such that $n_i =0$ for all but finitely many $i$, we write 
\[ \prod S^{n_i} a_i \in K_0(\Var / \Sym^{\sum n_i} X) \]
for the image of the obvious exterior product under the forgetful map
\[ K_0(\Var / \prod \Sym^{n_i} X_i) \rightarrow K_0(\Var / \Sym^{\sum n_i} X_i). \]
Finally, given $a \in K_0(\Var / \Sym^n X)$, we write $a^*$ for the image of $a$ under the composition of the restriction map
\[ K_0(\Var / \Sym^n X) \rightarrow K_0(\Var / \Conf^n X) \]
and the forgetful map
\[ K_0(\Var / \Conf^n X ) \rightarrow K_0(\Var / \CC). \]
The definition of the motivic Euler product \cite[3.8.1]{bilu:thesis-article}, combined with the identity (\ref{eqn:mot-euler-mot-power}) then gives:

\begin{theorem}  \label{thm:non-effective-power} 
If $b_\tau \in K_0(\Var/\CC),\;\tau \in \calQ_0$ is a collection of classes with $b_\emptyset=1$, and $a_\tau$ is the pullback of $b_\tau$ to $K_0(\Var/X)$, then
\[ \left( \sum_{\tau \in \calQ_0}  b_\tau \underline{t}^\tau \right)^{_\Pow [X]} = \sum_{\tau \in \calQ_0} \; \sum_{\substack{\tau = \sum n_{\tau'} \tau'\\\tau' \in \calQ_0,\; n_{\tau'} \in \bbZ_{\geq 0}}}  \left(\prod S^{n_{\tau'}} a_{\tau'}\right)^* \underline{t}^\tau.  \]
\end{theorem}
Theorem \ref{thm:non-effective-power} provides a systematic way to compute powers even when the base series is not effective. Using it, we prove Proposition \ref{prop:generating-identity}:

\begin{proof}[Proof of Proposition \ref{prop:generating-identity}]
We first observe that, for any $X$,
\begin{equation}\label{eqn:symmetric-power-neg} S^n (-[X/X]) = \sum_{\mu \in \calQ, |\mu|=n} (-1)^{||\mu||} [\Sym^\mu X / \Sym^{n} X]. \end{equation}
This follows from writing 
\[ \pi=S([X]) -1 = (0, [X/X], [\Sym^2 X / X], \ldots) \]
so that
\[ S([X])^{-1} = (1 + \pi)^{-1} = 1 - \pi + \pi^2 -\pi^3 + \ldots .\]

Now, applying Theorem \ref{thm:non-effective-power} (which is stated with variables $t_1, t_2, \ldots$, but which we use with variables $s, z, t_1, t_2, \ldots$), we find that
\begin{multline*} (1 - s\cdot(1+ t_1 + t_2 +...) + z \cdot (t_1 + t_2 + ...))^{_\Pow [X]} = \\
\sum_{\tau \in \calQ_0} \underline{t}^\tau \sum_{\substack{\tau = \tau_s + \tau_z \\ \tau_s, \tau_z \in \calQ_0}} \sum_{n \geq 0} s^{n + |\tau_s|} z^{|\tau_z|} \left(S^n(-[X/X]) \times \prod_i S^{\tau_s(i)}(-[X/X]) \times \prod_i S^{\tau_z(i)}([X/X]) \right)^* 
\end{multline*}
Substituting in (\ref{eqn:symmetric-power-neg}), we obtain
\[ 
\sum_{\tau\in \calQ_0} \underline{t}^\tau \; \sum_{\substack{\tau = \tau_s + \tau_z \\ \tau_s, \tau_z \in \calQ_0}} \; \sum_{n\geq 0} \sum_{\substack{|\mu'|=n, |\mu_i|=\tau_s(i), \\ \mu', \mu_i \in \calQ, }}  s^{n + |\tau_s|} z^{|\tau_z|} (-1)^{||\mu'|| + \sum_i ||\mu_i||} [\Conf^{\mu' \cdot \mu_1 \cdot \mu_2 \cdot \ldots  \cdot \tau_z} X]. \] 
Organizing the sum by $\mu = \mu' + \sum \mu_i$, and noting that the constraints on $\mu_i$ imply $|\mu'|+|\tau_s|=|\mu|$, we obtain
\[ 
\sum_{\tau\in \calQ_0} \underline{t}^\tau \sum_{\substack{\tau = \tau_s + \tau_z \\ \tau_s, \tau_z \in \calQ_0}}\; \sum_{\mu \in \calQ} \; \sum_{\substack{\mu=\mu' + \sum \mu_i,\\ \mu', \mu_i \in \calQ}}  s^{|\mu|} z^{|\tau_z|} (-1)^{||\mu'|| + \sum_i ||\mu_i||} [\Conf^{\mu' \cdot \mu_1 \cdot \mu_2 \cdot \ldots \cdot \tau_z} X]. \] 
Note that the configuration space appearing in this sum does not change if we replace the $\mu_i$ and $\mu'$, which are elements of $\calQ$, with partitions in $\calQ_0$ that lie in their preimage under the contraction operator $c$ introduced in \ref{subsec:partitions-and-combinatorial-lemma}. Thus, we can apply Lemma \ref{lemma:key-combinatorics-of-piles} to replace the summation over $\mu', \mu_i \in \calQ$ with a summation over $\mu', \mu_i \in \calQ_0$ to obtain
\[ 
\sum_{\tau\in \calQ_0} \underline{t}^\tau \sum_{\substack{\tau = \tau_s + \tau_z \\ \tau_s, \tau_z \in \calQ_0}} \; \sum_{\mu \in \calQ} (-1)^{||\mu||} \sum_{\substack{\mu=\mu' + \sum \mu_i,\\ \mu', \mu_i \in \calQ_0}}  s^{|\mu|} z^{|\tau_z|}  [\Conf^{\mu' \cdot \mu_1 \cdot \ldots \cdot \mu_{||\tau_s||} \cdot \tau_z} X]. \]
Substituting $\tau_1$ for $\tau_z$ and $\tau_2$ for $\tau_s$, and combining and reordering the sums, we obtain the right-hand side of (\ref{eqn:pro-gen-func}), as desired. 
\end{proof}

\subsection{Cohomological stabilization}
In this section we prove Theorem \ref{thm:StabCohomHodge}. The proof mirrors the deduction of Theorem \ref{thm:StabCohomPointCount} from Theorem \ref{thm:StableGeomPointCount} given in Subsection \ref{subsec:PCCohomStab}. As in the point counting case, our method also gives an algorithm to compute an explicit formula for the limit (cf. Appendix \ref{appendix:AlgAndComp} for a description of this algorithm).  

\begin{proof}[Proof of Theorem \ref{thm:StabCohomHodge}]
We use the notation developed in Subsection~\ref{subsec:ProbVHS} and Section~\ref{sec:vancohom}. We denote
\begin{multline*}
[Y^\old]_\GHS := \sum_{i < n-1} (-1)^i \cdot \left( [H^i(Y(\CC), \QQ)] +  [H^{2n-i}(Y(\CC), \QQ)(1)] \right) \\ 
+ (-1)^{n-1} [H^i(Y(\CC), \QQ)] \in K_0(\GHS).
\end{multline*}

We can view $[Y^\old]_\GHS$ as a constant virtual variation of Hodge structure on $U_d$. By Equation (\ref{eqn:VanCohomDSum-Hodge}) and the weak Lefschetz theorem, 
\[ [Y_\old]_\GHS + [\VanQ] = [Z_d / U_d]_{\GVHS}, \] 
and thus, using the additivity of $(p_k', \;)$, 
\begin{eqnarray*}
(p_k' - (p_k', [Y_\old]), [Z_d / U_d]_\GVHS ) & = & (p_k', [Z_d / U_d]_\GVHS) - (p_k', [Y_\old]) \\
& = & (p_k', [\VanQ] ).
\end{eqnarray*}
Denote by $s_{\tau, Y}$ the polynomial in $\Lambda_{K_0(\GHS)}$ obtained by expressing $s_\tau \in \Lambda$ as a polynomial in $p_k'$, then substituting $p_k' - (p_k', [Y_\old])$ for $p_k'$. By the above, we have
\[ (s_{\tau, Y}, [Z_d / U_d]_\GVHS) = (s_{\tau}, [\VanQ]) = [\VanQ^{\pi_{\tau,d}}]. \]
Thus,
\begin{eqnarray*}
\lim_{d\rightarrow \infty} \frac{\chi_\HS (H_c^\bullet(U_d(\CC), \VanQ^{\pi_{\sigma,d}}))}{[\QQ(-\dim U_d)]} &  = & \lim_{d\rightarrow \infty} \frac{\chi_\HS ([(s_{\tau, Y}, [Z_d / U_d]_\GVHS)]}{[\QQ(-\dim U_d)]}
\end{eqnarray*}

It suffices to verify that the limit on the right exists for any $f \in \Lambda$, since by linearity we can then extend the scalars on $\Lambda$ to $K_0(\GHS)$ to obtain the result for $s_{\tau, Y}$. 

For $f \in \Lambda,$ by Lemma \ref{lem:HodgeMotCommute},
\[ \lim_{d\rightarrow \infty} \frac{\chi_\HS ([(f, [Z_d / U_d]_\GVHS)]}{[\QQ(-\dim U_d)]} =
\lim_{d\rightarrow \infty} \frac{((f, [Z_d / U_d])_{\motcomp} )_{\HS} }{\LL^{\dim U_d}_\HS} \]

Applying the Hodge realization to Theorem \ref{thm:StabGeomMot}, we see the limit on the right exists (and has a universal formula of the form claimed).   

\end{proof}

\appendix
\section{Algorithms and explicit computations.}
\label{appendix:AlgAndComp}
In this appendix we give an algorithm for computing the explicit universal formulas in Theorem \ref{thm:StabCohomPointCount}, extracted from the proof in  Subsection \ref{subsec:PCCohomStab}. With minor modifications, this also computes the analogous limits in Theorem \ref{thm:StabGeomMot}.

\hfill\\
\noindent\textbf{Algorithm.}

\hfill\\
\noindent\textbf{Input:} A partition $\sigma$ and dimension $n$.

\hfill\\
\noindent\textbf{Output:} A universal formula computing
\[ \lim_{d \rightarrow \infty} q^{-\dim U_d}\sum_i (-1)^i \Tr \Frob_q \actsonr H_c^\bullet (U_{d,\overline{\FF_q}}, \VanQl^{\pi_{\sigma, d}}) \]
for polarized smooth projective varieties $Y/\FF_q$ of dimension $n$ in terms of symmetric functions of the eigenvalues of Frobenius acting on the cohomology of $Y$.

\hfill\\
\noindent\textbf{Step 1:} As $n-1=0$, is even positive, or odd, compute the appropriate Schur polynomial $s_\sigma$ of Proposition \ref{prop:SurjLambdaSchur} as a binomial polynomial in the $p_k'$. 

\hfill\\
\noindent\textbf{Step 2:} For each ``binomial monomial" 
\[ \binom{p}{\underline{l}} = \prod \binom{p_k'}{l_k} \]
appearing in this expression, 
substitute the coefficient of $\prod t_k^{l_k} $ appearing in 
\[ \prod_k \left( \left(1+ \frac{q^{nk} - 1}{q^{(n+1)k}-1} t_k \right)^{(p_k', [Y])} \cdot (1+t_k)^{-(p_k', [Y^\old])} \right)\]

\hfill\\
\noindent \textbf{Step 3:} Multiply the resulting expression by $\zeta_Y(n+1)^{-1}$. (Here we use that this is the alternating product of the characteristic series of Frobenius acting on the cohomology of $Y$, evaluated at $q^{-(n+1)}$). 

\begin{remark} Usually it is more convenient to omit Step 3, which corresponds to normalizing by $\#U_d(\FF_q)$ instead of $q^{-\dim U_d}$.  
\end{remark}

\begin{example}\label{example:AppdxPrimCohomVanish} For the standard representation, corresponding in all cases to the partition $(1)$, we obtain
\begin{eqnarray*} \EE_{\mu_\infty}[s_{\tau,Y}]& = &\EE_{\mu_\infty}[p_1' - (p_1', [Y^\old])] \\
& = & \frac{q^{n} - 1}{q^{n+1}-1}\# Y(\FF_q) - \Tr \Frob_q \actsonr [Y^\old]. \\
\end{eqnarray*}
For example, if $Y=\PP^{n}$ this gives,
\begin{eqnarray*} \EE_{\mu_\infty}[s_{\tau,\PP^{n}}] & = &\frac{q^{n} - 1}{q^{n+1}-1}\# \PP^{n}(\FF_q) - \#\PP^{n-1}(\FF_q) \\
& = & 0. 
\end{eqnarray*}
By Theorem \ref{thm:StabGeomMot}, we obtain a similar result for Hodge structures (cf. Example~\ref{example:IntroPrimCohomVanish}). 

We note that for the standard representation, $H^1$ is known to stabilize by Nori's connectivity theorem \cite[Corollary 4.4]{Nori-Connectivity}, and our results are compatible with the stable values that appear. 
\end{example}

\bibliography{references}
\bibliographystyle{plain}
\end{document}